\documentclass[12pt, twoside]{article}
\usepackage{amsmath,amsthm,amssymb}
\usepackage{times}
\usepackage{enumerate}

\def\titlerunning#1{\gdef\titrun{#1}}
\makeatletter
\def\author#1{\gdef\autrun{\def\and{\unskip, }#1}\gdef\@author{#1}}
\def\address#1{{\def\and{\\\hspace*{18pt}}\renewcommand{\thefootnote}{}%
\footnote {#1}}%
\markboth{\autrun}{\titrun}}
\makeatother
\def\email#1{e-mail: #1}
\def\subjclass#1{{\renewcommand{\thefootnote}{}%
\footnote{\emph{Mathematics Subject Classification (2010):} #1}}}
\def\keywords#1{\par\medskip
\noindent\textbf{Keywords.} #1}

\newtheorem{thm}{Theorem}[section]
\newtheorem{cor}[thm]{Corollary}
\newtheorem{lem}[thm]{Lemma}
\newtheorem{prop}[thm]{Proposition}
\newtheorem{question}[thm]{Question}

\theoremstyle{definition}
\newtheorem{defin}[thm]{Definition}
\newtheorem{rem}[thm]{Remark}
\newtheorem{exa}[thm]{Example}

\numberwithin{equation}{section}

\usepackage{amsfonts}
\usepackage{xy}
\usepackage{color}
\usepackage[dvipsnames]{xcolor}
\usepackage[colorlinks=true,linkcolor=RoyalBlue,citecolor=PineGreen,urlcolor=blue]{hyperref}
\xyoption{all}
\usepackage[numbers]{natbib}
\DeclareMathOperator{\ord}{ord}
\DeclareMathOperator{\id}{id}

\DeclareMathOperator{\codim}{codim}
\DeclareMathOperator{\Char}{char}

\begin{document}


\baselineskip=17pt


\titlerunning{Effective difference elimination and Nullstellensatz}

\title{Effective difference elimination and Nullstellensatz}

\author{Alexey Ovchinnikov
\and 
Gleb Pogudin \and Thomas Scanlon}

\date{}

\maketitle

\address{A. Ovchinnikov: CUNY Queens College, Department of Mathematics,
65-30 Kissena Blvd, Queens, NY 11367 and
CUNY Graduate Center, Ph.D. programs in Mathematics and Computer Science, 365 Fifth Avenue,
New York, NY 10016; \email{aovchinnikov@qc.cuny.edu}
\and
G. Pogudin: New York University, Courant Institute of Mathematical Sciences, New York, NY 10012; \email{pogudin@cims.nyu.edu}
\and
T. Scanlon: University of California at Berkeley, Department of Mathematics, Berkeley, CA 94720; \email{scanlon@math.berkeley.edu}}

\subjclass{Primary 12H10, 13P25; Secondary 14Q20, 03C10, 03C60}


\begin{abstract}
We prove effective 
Nullstellensatz and elimination
theorems for difference equations
in sequence rings.  
More precisely, we compute an
explicit function of geometric
quantities associated to a system 
of difference equations (and these
geometric quantities may themselves be 
bounded by a function of 
the  number of variables, the
order of the equations, and 
the degrees of the equations)
so that 
for any system of difference 
equations 
in variables $\mathbf{x} = (x_1, \ldots, x_m)$
and $\mathbf{u} = (u_1, \ldots, u_r)$, if these
equations have any nontrivial consequences
in the $\mathbf{x}$ variables, then such a 
consequence may be seen algebraically considering
transforms up to the order of our bound.  
Specializing to the case of $m = 0$, we obtain
an effective method to test whether a given system of difference equations is consistent.

\keywords{difference equations,
effective Nullstellensatz,
elimination of unknowns}
\end{abstract}

\section{Introduction}

Let $K$ be an algebraically closed field of arbitrary characteristic.  
We say that a sequence $(a_j)_{j =0}^\infty$ from $K$ satisfies a difference equation with constant
coefficients if there is a nonzero polynomial $F(x_0,\ldots,x_e) \in K[x_0,\ldots,x_e]$ such 
that, for every natural number $j$, the equation $F(a_j,a_{j+1},\ldots,a_{j+e}) = 0$ holds.  
This can also be defined for systems of difference equations in several variables.
Such difference equations and the sequences
that solve them are ubiquitous throughout mathematics and in its applications to the sciences, including 
such areas as combinatorics, number theory, control theory, 
and epidemiology, amongst many others (see Section~\ref{sec:NumEx} for some of the examples).   

In this paper we resolve some fundamental problems about difference equations.  The questions 
we answer include the following (for precise statements, including the way non-constant coefficients can appear, see Section~\ref{sec:main}):
\begin{enumerate}
\item Under what conditions does a system of difference equations have 
a sequence solution?  
\item Can these conditions be made sufficiently transparent to allow for efficient computation?  
\item 
Given a system of difference equations on $(n+m)$-tuples of sequences, how does one eliminate some of the variables so 
as to deduce the consequences of these equations on the first $n$ variables? 
\end{enumerate}
Our solution to the first 
question is a conceptual difference Nullstellensatz, to the second, an effective difference Nullstellensatz, and to 
the third, an effective difference elimination algorithm.   
 Even though the abstract Nullstellensatz
is intellectually satisfying in that conditions of different kinds
are shown to be equivalent, namely the existential condition that there
is a sequence solution to a system of difference equations and the 
universal condition that the difference ideal generated by the 
equations is proper, the difficult work and applications, both 
theoretical and practical, comes with our main effective 
theorems.    

Effective elimination theorems and methods
have a long history and play 
central roles in computational
algebra.  Row reduction, or 
Gaussian elimination, is a 
fundamental technique in linear 
algebra.  Elimination for 
polynomial equations is 
substantially more complicated
and has been the subject of 
intensive and sophisticated
work~\cite{Brownawell,Kollar,Jelonek}.  In recent work of the 
first two authors joined by 
Vo~\cite{OPV2017}, effective elimination 
theorems were obtained for 
algebraic differential 
equations through a reduction 
to the polynomial case through 
the decomposition-elimination-prolongation method.  
Elimination of unknowns for systems of linear difference equations is an essential part of the classical transfer matrix method in combinatorics~\cite[\S 4.7]{Stanley}.

While these questions are important and difference equations have been studied intensively both for their applications 
and  theory, to our knowledge, none of these questions has received a satisfactory answer 
in the literature.  We explain below how some known results, both positive and negative, may help explain 
the existence of this lacuna.  
In particular, in some 
essential ways, the effective
Nullstellensatz and elimination 
problems for difference equations
are substantially more difficult
than the corresponding problems 
for differential equations and the 
methods of~\cite{OPV2017} do not
routinely transpose to this context.

The foundational  
work on difference algebra, that is, the study of the theory of difference rings and of difference 
equations as encoded through the algebraic properties of rings of difference polynomials, was initiated by Cohn in~\cite{Cohn}, following
the tradition of Ritt and Kolchin in differential algebra.    Deep results have been obtained in this subject, but their relevance to the
problems at hand is hampered by their restrictions, for the Nullstellensatz and elimination theorems,  to the case in which
solutions are sought in difference \emph{fields}, and thus have little bearing on the structures used in practice, namely difference rings presented as rings of sequences, such 
as $\mathbb{C}^\mathbb{N}$ given with the shift operator $\sigma:(a_i)_{i = 0}^\infty 
\mapsto (a_{i+1})_{i =0}^\infty$. 
Moreover, even if restricted
to difference fields, the 
known elimination theorems are 
at best theoretically effective.

Chatzidakis and Hrushovski studied difference fields from the perspective of mathematical logic in~\cite{ChHr99}.  There,
they established a recursive axiomatization for the theory of existentially closed difference fields and proved a
quantifier simplification theorem.  From this it follows that in principle there are effective procedures to check 
the consistency of difference equations in difference fields and to perform difference elimination in difference fields. 
More recent work of Toma\v{s}i\'{c}~\cite{Tomasic,Tomasic2018} geometrizes the quantifier simplification theorem and brings the complexity
of these algorithms to primitive recursive, though this 
effectivity is still theoretical ---
to call the implicit bounds 
astronomical would be a gross 
understatement ---  and 
a practical implementation of 
this work is infeasible.  
In symbolic computation, 
steps have been taken 
towards extending the 
characteristic set method
from differential algebra 
to the study of difference 
and difference-differential 
equations in works of Gao, van der Hoeven, Li, Yuan, Zhang~\cite{GLY2009,Gao2009,LYG2015,LL15}.
These methods are more 
efficient than those coming 
from logic, but as they 
are restricted to the study of 
inversive prime difference 
ideals, they, too, are 
fundamentally results  about
solutions to difference equations
in difference fields and the 
constructions of difference 
resultants depend on 
restrictive hypotheses. 
A similar approach was taken by Lyzell, Glad, Enqvist, Ljung~\cite{LGEL2011} aiming at solving a problem in discrete-time control theory.

The situation for difference equations in sequence
rings differs starkly.  Simple examples show that consistency checking in difference fields is not the 
same problem as consistency checking for sequences.  For example, the system of difference equations 
$x \sigma(x) = 0,\ x + \sigma(x) = 1$ has no solution in a difference field, but the sequence $0, 1, 0, 1, 
\ldots$ is a solution in $\mathbb{C}^\mathbb{N}$.  

More seriously, theorems of Hrushovski and Point~\cite{HrPo}
show that the logical methods used for difference fields fail dramatically for sequence rings.  In particular, 
they show that the first-order theory of $\mathbb{C}^\mathbb{N}$ regarded as a difference ring is 
undecidable.  Thus, we cannot derive a consistency checking method from a recursive axiomatization of this 
theory nor can we produce an elimination algorithm from an effective quantifier elimination theorem; no such 
axiomatization or quantifier elimination procedure exists.  That we succeed in solving the effective consistency
checking and effective elimination problems for difference equations in sequence rings is all the more surprising
given these undecidability results.

Let us explain more precisely what we actually prove and where the new ideas appear in our arguments.  We have 
two main theorems: 
Theorem~\ref{th:nullstellensatz} an effective Nullstellensatz and Theorem~\ref{thm:main2} an 
effective difference elimination theorem.  Strictly speaking, 
the effective Nullstellensatz is a 
special case of an effective elimination theorem, but 
we prove elimination by 
bootstrapping through 
the Nullstellensatz.

The key to our work is a new proof technique based on the spirit of the 
decomposition-elimination-prolongation (DEP) method. 
As is completely standard,
a system of difference equations
may be regarded as a system of 
algebraic equations in more
variables together with 
specifications that certain 
coordinates should be 
obtained from others by 
the application of the 
distinguished endomorphism and 
the usual DEP methods allow 
for one to cleverly reduce 
questions about the original 
system of difference equations
to questions entirely about 
algebraic equations.   
A version of the DEP method for difference equations 
in difference fields is 
employed in~\cite{Hr01} for 
the purpose of computing 
explicit bounds in Diophantine
geometric problems.   
This DEP method cannot work for the 
problems at hand as explained in Section~\ref{sec:counterexamples}.
We overcome this obstacle by  taking a different approach to reducing the question about the original system to the question about algebraic equations. 
The core of this reduction is for us to show that every system of difference equations
that has a solution actually has what we call a skew-periodic solution with the components being (not necessarily closed!) points of the affine variety corresponding to the original system, and the length of the period can be bounded in terms of the geometric data of the original system (see Section~\ref{subsec:bound_for_trains}).

With our theorems we 
explicitly bound the number
of prolongations required 
to solve the problems 
at hand, \emph{i.e.} testing 
a system of difference equations
for consistency or computing 
a nontrivial element of the
elimination ideal.  
For the elimination problem, our bound is not sensitive to the number of variables that are not being eliminated, see Remark~\ref{rem:smallbound}.
The bounds are 
 small enough in many cases 
 to permit efficient computation, see Section~\ref{sec:NumEx}.

We draw an interesting theoretical conclusion from 
our work towards the explicit bounds for the 
difference elimination problem in Section~\ref{sec:smallfields}.  
Specifically, with Theorem~\ref{thm:DNSS-small}, we show that for 
$(K,\sigma)$ any algebraically closed difference field, whenever a finite
system of difference equations over $K$ is consistent in the sense that
it has a solution in some difference ring, then it already 
has a solution in the ring of sequences of elements of $K$.  
 We 
give a soft proof of such a difference Nullstellensatz under the
hypothesis that $K$ is uncountable with Proposition~\ref{thm:DNSS}. The 
proof of Theorem~\ref{thm:DNSS-small} is much more difficult than it may have 
been expected to be.
In extending this difference Nullstellensatz to general $K$ we use crucially our result that
a system of difference equations is consistent if and only if it has a 
skew-periodic solution and then appeal to remarkable theorems of Hrushovski on 
the first-order theory of the Frobenius automorphism and of Varshavsky on 
intersections of correspondences with the graph of the Frobenius.

The paper is organized as follows. We give the basic definitions in Section~\ref{sec:general}, and then introduce the notation and terminology specific to our paper.
The main results, Theorem~\ref{th:nullstellensatz} for the effective Nullstellensatz and Theorem~\ref{thm:main2} for the effective elimination, are 
expressed in Section~\ref{sec:main}.
In Section~\ref{sec:NumEx}, we
illustrate our results 
in several practical examples.
With Section~\ref{sec:counterexamples},
we present counterexamples to
an effective strong difference
Nullstellensatz and to the 
application of the 
usual DEP method to these 
problems.  The proofs
of the main theorems 
are presented in 
Section~\ref{sec:proof}.
Finally, in Section~\ref{sec:smallfields}, we
strengthen the difference
Nullstellensatz giving 
equivalent criteria for the 
existence of sequence solutions 
to systems difference equations
over any algebraically closed field.

\section{Preliminaries}\label{sec:general}
Throughout the paper, $\mathbb{N}$ denotes the set of non-negative integers.
A detailed introduction to difference rings can be found in \cite{Cohn,LevinBook}.
\begin{defin}[Difference rings]
\begin{itemize}
\item[]
\item 
A {\em difference ring} is a pair $(A,\sigma)$ where
$A$ is a commutative ring and $\sigma:A \to A$
is a ring endomorphism.  

 \item As an example, if $R$ is any commutative ring, then the sequence rings $R^{\mathbb{N}}$ and $R^{\mathbb{Z}}$ are difference rings with 
  $\sigma$ defined by $\sigma ( (x_i)_{i \in \mathbb{N}}) := ( x_{i+1} )_{i \in \mathbb{N}}$ 
  ($\sigma ( (x_i)_{i \in \mathbb{Z}}):= (x_{i+1} )_{i \in \mathbb{Z}}$, respectively).
\item 
A {\em map of difference rings} $\psi:(A,\sigma) 
\to (B,\tau)$ is given by a map of rings $\psi:A \to B$ 
such that
that $\tau \circ \psi = \psi \circ 
\sigma$.   
\item We often abuse notation saying that $A$ is a difference
ring when we mean the pair $(A,\sigma)$.   
\end{itemize}
\end{defin}

\begin{defin}[Difference polynomials]\label{def:diff_poly}
Let $A$  
be
a difference ring.
\begin{itemize}
\item The free difference $A$-algebra in 
one generator $x$ over $A$, $A \{ x \}$, also called the {\em ring 
of difference polynomials} in $x$ over $A$, may be realized as the ordinary
polynomial ring $A[ \{ \sigma^j(x) ~:~ j \in \mathbb{N}\}]$ in the 
indeterminates $\{ \sigma^j(x) ~:~ j \in \mathbb{N} \}$. 
\item Iterating 
this procedure, one obtains the difference polynomial ring $A \{ x_1, 
\ldots, x_n \}$ in $n$ variables. 
\item 
Every difference polynomial in $A\{x_1, \ldots, x_n\}$ can be considered as an ordinary polynomial in indeterminates of the form $\sigma^i(x_j)$.
\item For $P \in A \{ x_1, \ldots, x_n \}$
and $1 \leqslant i \leqslant n$, we define the {\em order} of $P$ with respect to $x_i$,
denoted $\operatorname{ord}_{x_i}(P)$ to be the maximal $h$ for which 
$\sigma^h(x_i)$ appears in $P$.  If no $\sigma^h(x_i)$ appears, we set
$\operatorname{ord}_{x_i}(P) := -1$.  
We also set $\ord P := \max\limits_{1 \leqslant i \leqslant n} \ord_{x_i} P$.
\end{itemize}
\end{defin}

\begin{exa} $\ord_{x_3}(\sigma^3(x_1)+x_2+\sigma(x_3)^2+1$)=1.
\end{exa}
\begin{defin}
If $(A,\sigma)$ is a difference ring and $F \subseteq A \{ x_1, \ldots, x_n \}$
is a set of difference polynomials over $A$,
$(A,\sigma) \subseteq (B,\sigma)$ is an extension of difference rings, and 
$\mathbf{b} = (b_1, \ldots, b_n) \in B^n$ is an $n$-tuple from $B$, then we 
say that $\mathbf{b}$ is a {\em solution} of the
system $F = 0$ if,
under the unique map of difference rings $A \{ x_1, \ldots, x_n \} \to B$
given by extending the given map $A \to B$ and sending $x_i \mapsto b_i$ for 
$1 \leqslant i \leqslant n$, every element of $F$ is sent to $0$.  
\end{defin}

\begin{exa} Let $(A,\sigma)=(\mathbb{Q},\id)$ and $(B,\sigma)=(\mathbb{Q}^\mathbb{N},\sigma)$, where $\sigma$ is the shift (to the left) operator. Then the tuple \[\mathbf{b} = ((1,0,1,0,\ldots),(2018,1,0,1,\ldots)) \in B^2 \] 
is a solution of the 
system
\[
\begin{cases}
\sigma(x_1)+ x_1-1=0,\\
\sigma(x_2)-x_1=0.
\end{cases}
\]
\end{exa}

\begin{defin}\label{def:transform} If $(A,\sigma)$ is a difference ring,
$F \subseteq A \{ x_1, \ldots, x_n \}$, and $B$ is a non-negative integer, the   {\em $B$-th transform} of $F$ is the set \[\sigma^B(F) := \left\{\sigma^B(f)\:|\: f\in F\right\}.\] So, the $0$-th transform of $F$ is $F$.
The $B$-th transform of a system of difference equations is defined similarly.
\end{defin}

\begin{exa}
The 2-nd transform of the system
\[
\begin{cases}
\sigma(x_1)^5=x_1+x_2^2\\
x_3^3+x_1+1=0
\end{cases}
\]
is the system
\[
\begin{cases}
\sigma^3(x_1)^5=\sigma^2(x_1)+\sigma^2(x_2)^2\\
\sigma^2(x_3)^3+\sigma^2(x_1)+1=0.
\end{cases}
\]
\end{exa}

The ideal generated by a set $F$ in a commutative ring $R$ is denoted by $\langle F \rangle$.

\begin{defin} A difference equation $g(x_1, \ldots, x_n) = 0$ is said to be a {\em consequence} of a system  of difference equations $F = 0$, where $F \subset k\{x_1,\ldots,x_n\}$, if there exists a non-negative integer $B$ such that 
\[
g \in \big\langle\sigma^i(F)\:|\: 0\leqslant i< B\big\rangle.
\]
\end{defin}
\begin{exa}
Let $F = 0$ be the system
\[
\begin{cases}
f_1 = x_2\sigma(x_1) - x_1 - 1 = 0\\
f_2 = \sigma(x_2) - x_2^2 = 0.
\end{cases}
\]
The equation $\sigma^2(x_1)x_2^2 - \sigma(x_1) - 1$ is a consequence of $F = 0$  with 
$B = 2$
because
\[
\sigma(f_1) - \sigma^2(x_1) f_2 = \sigma(x_2\sigma(x_1)-x_1-1)-\sigma^2(x_1)(\sigma(x_2)-x_2^2)=\sigma^2(x_1)x_2^2-\sigma(x_1)-1.
\]
\end{exa}

We define the degree of an affine algebraic variety following~\cite[Definition~1 and Remark~2]{Heintz} as follows.
\begin{defin}\label{def:deg}
  Let $X$ be an irreducible affine variety of dimension $r$ in $\mathbb{A}^n$.
  Then we define
  \[
  \deg X := \max\big\{ |X \cap E| \;|\; E \text{ is an affine subspace of $\mathbb{A}^n$ with } \dim E = n - r\text{ and }|X \cap E| < \infty\big\}.
  \]
  
  Let $X$ be an affine variety defined over a field $k$.
  Let $X = X_1 \cup \ldots \cup X_N$ be the decomposition of $X$ into irreducible components over the algebraic closure of $k$.
  Then we define
  \[
  \deg X := \sum\limits_{i = 1}^N \deg X_i.
  \]
\end{defin}

\section{Main results}
\label{sec:main}

For all $d \in \mathbb{Z}_{\geqslant 0}$ and $D \in \mathbb{Z}_{> 0}$ we define
\[
B(d, D) = \begin{cases}
  D + 1& \text{ if } d = 0,\\
 \frac{D^3}{6} + \frac{D^2}{2} + \frac{4D}{3} + 1& \text{ if } d = 1,\\
  B(d - 1, D) + D^{B(d - 1, D)}& \text{ if } d > 1.
\end{cases}
\]
\subsection{Effective difference Nullstellensatz}
\label{subsec:nullstellensatz}
\begin{thm}\label{th:nullstellensatz}
  Let 
  \begin{itemize}
  \item 
  $k$ be a difference field and $F = 0$ a system of difference equations, where $F := \{f_1, \ldots, f_N \} \subset k \{ u_1, \ldots, u_r \}$.
  \item 
  We set 
  \[
  	h_i := \max\limits_{j = 1, \ldots, N} \ord_{u_i} f_j\quad \text{and}\quad H = h_1 + \ldots + h_r  + r,
  \]
  so, $H$ is an upper bound on the number of the
  $\mathbf{u}$-unknowns and their transforms that appear in $F$.
  \item 
   $d(F)$ and $D(F)$ denote the dimension and degree of 
  the affine variety defined by $F$ over $k$ in the affine $H$-space, respectively.
  \end{itemize}
 The following statements are equivalent:
  \begin{enumerate}
  \item The system $F=0$ has a solution in a difference ring containing $k$;
  \item
  The system  $\{\sigma^i(F)=0\mid 0\leqslant i <B(d, D)\}$
   is consistent as a system of
  polynomial equations.
 \end{enumerate}
\end{thm}

\begin{cor}\label{cor:main}
If $k=\mathbb{C}$ in Theorem~\ref{th:nullstellensatz}, then  the following statements are equivalent: 
\begin{enumerate}
\item
The system $F=0$ has a solution in  $\mathbb{C}^\mathbb{Z}$ ;
  \item The system $\{\sigma^i(F)=0\mid 0\leqslant i <B(d, D)\}$
   has a solution in $\mathbb{C}$ as a system of
  polynomial equations.
\end{enumerate}
\end{cor}

\begin{rem} We do not prove an  effective strong Nullstellensatz generalizing Corollary~\ref{cor:main}, because such a statement is false as shown in Section~\ref{subsec:strong_counterexample}.
\end{rem}

\subsection{Effective elimination}
\label{subsec:effelim}
We will introduce the notation that will be used in Theorem~\ref{thm:main2}.
\begin{itemize}
\item Let $\mathbf{x} = (x_1, \ldots, x_{m})$ and $\mathbf{u} = (u_1, \ldots, u_{r})$ be two sets of unknowns.
\item Consider a system $F = 0$ of difference equations, where $F := \{ f_1, \ldots, f_N\} \subset k\{\mathbf{x}, \mathbf{u}\}$. 
We would like to have an effective method for determining whether there exists a nonzero consequence of the system $F = 0$ involving only the $\mathbf{x}$-variables.
\item 
We set 
\[
 h_i := \max\limits_{j = 1, \ldots, N} \ord_{u_i} f_j \quad\text{and}\quad H = h_1 + \ldots + h_r  + r,
\]
so, $H$ is an upper bound on the number of the $\mathbf{u}$-unknowns and their transforms that appear in $F$.
\item 
Let $E$ be the field of fractions of $k \{ \mathbf{x} \}$ and $X$ denote the associated affine subvariety of $\mathbb{A}^H$ defined by $F = 0$ over $E$.
Note that $X$ is not necessarily irreducible.
\item 
We denote the dimension and degree of $X$ by $d_{\mathbf{u}}(F)$ and $D_{\mathbf{u}}(F)$, respectively.
\end{itemize}

\begin{thm}
  \label{thm:main2} 
  For all integers $d \geqslant0$ and $D\geqslant 1$ and systems $F=0$ in $\mathbf{x}$ and $\mathbf{u}$
  with $d_{\mathbf{u}}(F) = d$ and $D_{\mathbf{u}}(F)=D$,
  the following statements are equivalent:
  \begin{enumerate}
  \item There exists a non-zero difference equation $g(\mathbf{x}) = 0$ that is a consequence of the system $F = 0$;
  \item
   $\big\langle \sigma^i(F)\mid 0\leqslant i < B(d, D)\big\rangle\cap k\{\mathbf{x}\}\ne \{0\}$.
  \end{enumerate}
\end{thm}

\begin{rem}
  Based on the existing elimination results for differential-algebraic equations~\cite[Theorem~3]{OPV2017}, it is tempting to find, for a positive integer $h$, a bound $B$ such that the ideal 
  \[
    \big\langle \sigma^i(F)\mid 0\leqslant i < B\big\rangle
  \]
  contains all the consequences of the system $F = 0$ depending only on $\mathbf{x}$-variables of order at most $h$.
  However, as we show in Section~\ref{subsec:no_full}, there is no such bound in terms of degrees, orders, and the number of variables. 
  Moreover, every such bound will depend on the coefficients of $F$.
\end{rem}

\begin{rem}\label{rem:smallbound}
  The bound in Theorem~\ref{thm:main2} is especially small if the number of the variables to eliminate is moderate. More precisely, $d\leqslant H - 1$, and $D$ does not exceed the product of the degrees of $H + 1$ equations of the highest degree.
  For particular examples, see Section~\ref{sec:NumEx}.
\end{rem}

\subsection{Consequences for computation} Theorem~\ref{th:nullstellensatz} and Corollary~\ref{cor:main} reduce consistency questions for systems of difference equations to consistency questions (in algebraically closed fields) of polynomial systems in finitely many variables and Theorem~\ref{thm:main2} reduces the question of existence/finding a consequence in the 
$\mathbf{x}$ variables of a system 
of difference equations in
the variables $\mathbf{x}$ and 
$\mathbf{u}$ to a question about a polynomial ideal in a polynomial ring in finitely many variables. 
These algebraic problems 
are classical and have been
computationally solved using, for example, Gr\"obner bases, triangular sets, numerical algebraic geometry, etc.
For all of these methods, implementations exist in many computer algebra systems and independent software packages (see, for example, \cite{CLO,Bertini,STV2008}).

\section{Numerical values and practical examples}\label{sec:NumEx}
In the following table, we compute $B(d, D)-1$ for small $d$ and $D$.
\begin{center}
\begin{tabular}{|c|c|c|c|c|c|c|}
\hline
  $d \setminus D$ & $1$& $2$ & $3$ & $4$ &$5$ \\ \hline
  $0$ & $1$ & $2$ & $3$ & $4$ &$5$\\ \hline
  $1$ & $2$ & $6$ & $13$ & $24$ & $40$\\ \hline
\end{tabular}
\end{center}
\begin{rem}\label{rem:balanced}Almost all examples of modeling phenomena in the sciences using polynomial difference equations that we have seen in the literature can be written as systems with the same number of equations as unknowns in such a way that none of the equations is a consequence of the others. The above table is applicable to elimination problems for such systems with $n$ equations if the problem is to eliminate $\lceil n/2\rceil$ unknowns or less, as such problems typically result in varieties $X$ (see the notation of Section~\ref{subsec:effelim}) of dimension $0$ or $1$.
\end{rem}
\begin{rem}\label{rem:4}One can significantly speed up checking if an elimination is possible by 
\begin{enumerate}
\item Applying the number of transforms that is in the bound; \item Substituting random values into the variables that are not being eliminated.
\end{enumerate}
Using techniques from~\cite[Section~5]{OPV2017} (see also~\cite{HOPY2017}) based on the DeMillo-Lipton-Schwartz-Zippel lemma \cite[Proposition~98]{Zippel},
for each number $p$, $0<p<1$,  
we can find the range for the random substitution so that the probability of the elimination being possible if and only if the ``substituted'' system has no solutions is greater than $p$. So, this would give an efficient probabilistic test for the possibility of elimination. 
\end{rem}

\begin{rem} Although  there could be special tricks and methods for each of the examples below, our approach provides a general and fully automated procedure.
\end{rem}

\begin{exa}\label{ex:ML2}
Consider the May-Leonard model for 2-plant annual competition, scaled down from \citep{RA2004}:
\[
\begin{cases}
x_{n+1}=\frac{(1-b)x_n}{x_n+\alpha_1y_n}+bx_n,\\
y_{n+1}=\frac{(1-b)y_n}{\alpha_2x_n+y_n}+by_n,\\
\end{cases}
\]
which can be
rewritten as
\begin{equation}\label{eq:ML2sigma}
\begin{cases}
(x+\alpha_1y)\sigma(x)=(1-b)x+bx(x+\alpha_1y),\\
(\alpha_2x+y)\sigma(y)=(1-b)y+by(\alpha_2x+y),\\
\end{cases}
\end{equation}
where 
$k = \mathbb{Q}(\alpha_1,\alpha_2,b)$, 
with $\sigma$ acting as the identity on $k$. 
To verify whether $y$ can be eliminated from~\eqref{eq:ML2sigma}, we then consider the affine variety $X$ defined by~\eqref{eq:ML2sigma}
over the field $\mathbb{Q}(\alpha_1,\alpha_2,b,x,\sigma(x))$ with coordinates $y,\sigma(y)$. A computation shows that $d=0$ and $D = 1$, and so $B(d,D)-1=2-1=1$.  
A computation shows that it is not only sufficient but also necessary to apply this single transform to perform the elimination. So, our main result gives a sharp upper bound for this example.
\end{exa}
\begin{exa}
Consider the May-Leonard model for 3-plant annual competition \citep{RA2004}:
\[
\begin{cases}
x_{n+1}=\frac{(1-b)x_n}{x_n+\alpha_1y_n+\beta_1z_n}+bx_n,\\
y_{n+1}=\frac{(1-b)y_n}{\alpha_2x_n+y_n+\beta_2z_n}+by_n,\\
z_{n+1}=\frac{(1-b)z_n}{\alpha_3x_n+\beta_3y_n+z_n}+bz_n,\\
\end{cases}
\]
which can be
rewritten as 
\begin{equation}\label{eq:ML3sigma}
\begin{cases}
(x+\alpha_1y+\beta_1z)\sigma(x)=(1-b)x+bx(x+\alpha_1y+\beta_1z),\\
(\alpha_2x+y+\beta_2z)\sigma(y)=(1-b)y+by(\alpha_2x+y+\beta_2z),\\
(\alpha_3x+\beta_3y+z)\sigma(z)=(1-b)z+bz(\alpha_3x+\beta_3y+z),
\end{cases}
\end{equation}
where 
$k = \mathbb{Q}(\alpha_1,\alpha_2,\alpha_3,\beta_1,\beta_3,\beta_3,b)$, 
with $\sigma$ acting as the identity on $k$. 
To verify whether $y$ and $z$ can be eliminated from~\eqref{eq:ML3sigma}, we consider the affine variety $X$ defined by~\eqref{eq:ML3sigma}
over the field $\mathbb{Q}\left(\alpha_1,\alpha_2,\alpha_3,\beta_1,\beta_3,\beta_3,b,x,\sigma(x)\right)$ with coordinates $y,\sigma(y),z,\sigma(z)$. A computation shows that
$d=1$ and $D=3$, and so $B(d,D)-1= 13$. A computation shows that
\begin{itemize}
\item two prolongations are necessary and sufficient
\item carrying out a computation with $13$ transforms and probability $p = 0.99$ as described in Remark~\ref{rem:4} to check if an elimination is possible does not take significantly more time than doing this with two transforms.

\end{itemize}
\end{exa}
\begin{exa}\label{ex:leslie_gower} Consider the  stage structured Leslie-Gower model \cite[eq.~(5)]{CHR07}:
\[
\begin{cases}
J_{n+1}=b_1\frac{1}{1+d_1A_n}A_n\\
A_{n+1}=s_1\frac{1}{1+J_n+c_1j_n}J_n\\
j_{n+1}=b_2\frac{1}{1+d_2a_n}a_n\\
a_{n+1}=s_2\frac{1}{1+c_2J_n+j_n}j_n,
\end{cases}
\]
which can be rewritten as
\begin{equation}\label{eq:LG}
\begin{cases}
(1+d_1A)\sigma(J)=b_1A\\
(1+J+c_1j)\sigma(A)=s_1J\\
(1+d_2a)\sigma(j)=b_2a\\
(1+c_2J+j)\sigma(a)=s_2j,
\end{cases}
\end{equation}
where 
$k = \mathbb{Q}(b_1,b_2,c_1,c_2,d_1,d_2,s_1,s_2)$ 
with $\sigma$ acting as the identity on $k$.  
To verify whether $J$ and $j$ can be eliminated from~\eqref{eq:LG}, we consider the affine variety $X$ defined by~\eqref{eq:LG}
over the field $\mathbb{Q}(b_1,b_2,c_1,c_2,d_1,d_2,s_1,s_2,a,A,\sigma(a),\sigma(A))$ with coordinates $j,\sigma(j),J,\sigma(J)$. A computation shows that $d=0$; $D=1$ as the equations are linear in $j,\sigma(j),J,\sigma(J)$. Then \[B(d,D)-1=2-1=1.\] A computation shows that it is not only sufficient but also necessary to apply this single transform to perform the elimination. So, our main result gives a sharp upper bound for this example.
\end{exa}

\begin{exa}
A discrete multi-population {\em SI} model from \cite{Allen1994}, similarly to the previous examples, can be rewritten as
\begin{equation}\label{eq:SI}
\begin{cases}
\sigma(S)=S\left(1-\frac{a\cdot\Delta t}{N_1} I-\frac{b\cdot\Delta t}{N_1} i\right)\\
\sigma(s)=s\left(1-\frac{c\cdot\Delta t}{N_2} I-\frac{d\cdot\Delta t}{N_2} i\right)\\
\sigma(I)=I+S\left(\frac{a\cdot\Delta t}{N_1} I+\frac{b\cdot\Delta t}{N_1} i\right)\\
\sigma(i)=i+s\left(\frac{c\cdot\Delta t}{N_2} I+\frac{d\cdot\Delta t}{N_2} i\right),
\end{cases}
\end{equation}
where  
$k = \mathbb{Q}(a,b,c,d,\Delta t,N_1,N_2)$ 
with $\sigma$ acting as the identity on $k$. 
\begin{itemize}
\item
To verify whether $I,i$ can be eliminated from~\eqref{eq:SI}, we consider the affine variety defined by~\eqref{eq:SI} over $\mathbb{Q}(a,b,c,d,\Delta t,N_1,N_2,s,\sigma(s),S,\sigma(S))$, and so  $d=0, D=1$, thus \[B(d,D)-1=2-1=1.\]
\item
To verify whether $I,i,s$ can be eliminated from~\eqref{eq:SI}, we consider the affine variety defined by~\eqref{eq:SI}
over $\mathbb{Q}(a,b,c,d,\Delta t,N_1,N_2,S,\sigma(S))$, and so $d=2, D=2$. We compute
\[
  B(2, 2) - 1 = 135 - 1 = 134.
\] 
It turns out to be computationally feasible to carry out a computation with $134$ transforms and probability $p = 0.99$ as described in Remark~\ref{rem:4}. 
The output of the computation is that the elimination is possible.
\end{itemize}
\end{exa}

\begin{exa}\label{ex:fibonacci}
  Let $F_n$ be the $n$-th Fibonacci number.
  It turns out~\cite[p. 856]{Zeilberger2014} that the sequence $A_n := F_{2^n}$ 
  satisfies a nonlinear difference equation.
  Such an equation can be found using difference elimination as follows.
  We introduce $B_n := F_{2^n + 1}$.
  Then standard identities $F_{2k} = F_k(2F_{k + 1} - F_k)$ and $F_{2k + 1} = F_{k + 1}^2 + F_k^2$ for the Fibonacci numbers imply the following system of difference equations
\begin{equation}\label{eq:Fib}
  \begin{cases}
    A_{n + 1} = A_n(2B_n - A_n),\\
    B_{n + 1} = A_n^2 + B_n^2.
  \end{cases}
\end{equation}
  Considered as a system of polynomial equations in $B_n$ and $B_{n + 1}$,~\eqref{eq:Fib} defines an affine variety of dimension zero and degree two
  over $\mathbb{Q}(A_n, A_{n + 1})$.
  Theorem~\ref{thm:main2} implies that it is sufficient to consider system~\eqref{eq:Fib} and two of its transforms to eliminate $B$.
  Performing this elimination, we find the difference equation
  \[
    5 F_{2^n}^4 F_{2^{n + 1}} - 2 F_{2^n}^2 F_{2^{n + 2}} + F_{2^{n + 1}}^3 = 0,
  \]
  giving an alternative to the difference equation stated in~\cite[p.~856]{Zeilberger2014}.
Our approach to finding a difference equation for $F_{2^n}$ can be viewed as a generalization of the transfer matrix method~\cite[\S 4.7]{Stanley} to the case of nonlinear recurrences.
\end{exa}

\begin{exa}
  The following example shows that our bound is sharp in the case $d = 0$ (this is the case in Examples~\ref{ex:ML2}, \ref{ex:leslie_gower}, and~\ref{ex:fibonacci}).
  We fix a positive integer $D$ and consider the system
  \begin{equation}\label{eq:sharp_d0}
  \begin{cases}
    x (x - 1) \cdot\ldots\cdot (x - D + 1) = 0,\\
    \sigma(x) - x - 1 = 0.
  \end{cases}
  \end{equation}
  System~\eqref{eq:sharp_d0} does not have a solution in~$\mathbb{C}^{\mathbb{Z}}$, because the elements of the solution can only take  values from $0, 1, \ldots, D - 1$ and strictly increase.
  On the other hand, the system consisting of the $0$-th,$\ldots,D - 1 = (B(0, D) - 2)$-th transforms of~\eqref{eq:sharp_d0} 
  has a solution $\sigma^i(x) = i$ for $0 \leqslant i \leqslant D$.
  Hence, it is necessary to consider one more transform in order to express $1$ (i.e. eliminate $x$).
\end{exa}
\begin{exa}
This example obtained by analyzing the proof of Proposition~\ref{prop:bound1} shows that our bound is sharp for $d = 1$ and $D = 2$. 
Consider a system of difference equations given by any set of generators of the polynomial ideal $I := I_1 \cap I_2$ of the polynomial ring $\mathbb{Q}[x, \sigma(x), y, \sigma(y)]$, where
\[
  I_1 := \langle x, \sigma(x) + \sigma(y) - 1, y + 2\sigma(y) - 1\rangle, \quad I_2 := \langle \sigma(x), y, x + 3\sigma(y) - 1 \rangle.
\]
The variety $X$ defined by $I$ is a union of two affine subspaces of dimension one, so $d = \dim X = 1$ and $D = \deg X = 2$.
Thus, $B(d, D) - 1 = 6$.
Our computation in {\sc Maple} shows that
\[
  1 \in \langle I, \sigma(I), \ldots, \sigma^6(I) \rangle \quad\text{ but }\quad 1\not\in \langle I, \sigma(I), \ldots, \sigma^5(I) \rangle.
\]
Thus, our bound for $d = 1$ and $D = 2$ is sharp.
\end{exa}

\section{Counterexamples}\label{sec:counterexamples}
\subsection{Failure of the standard DEP method} Consider the system of difference equations given by any set of generators of the polynomial ideal $I := I_1\cap I_2$ of the polynomial ring $\mathbb{Q}[x,\sigma(x),y,\sigma(y),z,w]$,
where
\[
I_1 := \langle \sigma(y) z - 1, x, \sigma(x) - y\rangle ,\quad I_2 := \langle \sigma(x), \sigma(y) - 1, (y - 1) z - 1, (x - 1) w - 1\rangle.
\]
We do not present the actual generators of $I$ due to the size of this set, the generators can be computed by a computer algebra system such as  {\sc Maple}. 
A computation in {\sc Maple} shows that 
\[
1 \in \left\langle I, \sigma(I), \sigma^2(I), \sigma^3(I), \sigma^4(I) \right\rangle.
\] 
Therefore, by Proposition~\ref{thm:DNSS}, the system has no solutions in any difference ring.
Using {\sc Maple}, one can also verify that
\begin{align}
I &= \langle I, \sigma(I)\rangle\cap \mathbb{Q}[x,\sigma(x),y,\sigma(y),z,w],\label{eq:proj1}\\ \sigma(I) &= \langle I,\sigma(I)\rangle\cap \mathbb{Q}[\sigma(x),\sigma^2(x),\sigma(y),\sigma^2(y),\sigma(z),\sigma(w)].\label{eq:proj2}
\end{align}
Most of the existing effective bounds for systems of ordinary differential and difference equations \cite{Gal,AJS,Hr01,Hrushovski-Pillay-bounds,OPV2017} use sufficient conditions for the existence of a solution based on the system and its first prolongation (differential equations) or first transform (difference equations), introduced for difference equations in \cite[Section~14, Chapter~8]{Cohn} and also known as geometric axioms \cite{ChHr99,PP1998} in model theory, which are summarized under the DEP method mentioned in the introduction.
In our case, it is tempting to formulate an analogue of such conditions as: 
\begin{quote}Let $\Gamma$ be the affine variety defined by the system and its first transform.
If the projections of $\Gamma$ onto the varieties defined by the system and by its first transform alone, respectively, are dominant, then the system is consistent.
\end{quote}
However, this is false in the above example as we have shown, where $\Gamma$ is the affine variety corresponding to the ideal $\langle I,\sigma(I)\rangle$ in the affine space with coordinates $x,\sigma(x),\sigma^2(x),y,\sigma(y),\sigma^2(y),z,\sigma(z),w,\sigma(w)$, and (the Zariski closures of) the projections are given by the intersections in~\eqref{eq:proj1} and~\eqref{eq:proj2}.

\subsection{Non-existence of coefficient-independent effective strong Nullstellensatz}\label{subsec:strong_counterexample}

  A (non-effective) strong  Nullstellensatz for systems of difference equations can be stated as follows.
  Let $f_1 = \ldots = f_N = 0$  be a system of difference equations.
  If a difference polynomial $f$ vanishes at all solutions of the system in $\mathbb{C}^{\mathbb{N}}$, 
  then there exists $\ell$ such that $f$ belongs to the radical of the ideal generated by the $0$-th,$\ldots,\ell$-th 
  transforms of $f_1, \ldots, f_N$.
      
  The following example shows that there is no uniform upper bound for this $\ell$ in terms of the degree, 
  order, and number of variables of $f_1, \ldots, f_N$.
 For every positive integer $M$, consider
  \begin{equation}\label{eq:strong_counterexample}
  \begin{cases}
    f_1 = \sigma(x) - x - \frac{1}{M} = 0,\\
    f_2 = x\left( y(x - 1)  - 1 \right) = 0.
  \end{cases}
  \end{equation}
  Let $f = y(x - 1) - 1$ and 
  $x = \{ x_n \}_{n = 0}^{\infty}$ 
  and  
  $y = \{ y_n \}_{n = 0}^{\infty}$ 
  any solution 
  of~\eqref{eq:strong_counterexample} in 
  $\mathbb{C}^{\mathbb{N}}$.
  If $y_k(x_k - 1) - 1 \neq 0$ for some $k$, then $x_k = 0$.
  Hence, $x_{k + M} = 1$, and so
  \[
  x_{k + M} \left( y_{k + M}(x_{k + M} - 1) - 1 \right) = -1.
  \]
  Therefore, $f$ vanishes at every solution of~\eqref{eq:strong_counterexample} in $\mathbb{C}^{\mathbb{N}}$.
  However, $f$ does not belong to the radical of the ideal generated by the $0$-th,$\ldots,(M - 1)$-th transforms 
  of $f_1$ and $f_2$.
  These transforms belong to the polynomial ring $\mathbb{C}[x, \ldots, \sigma^{M}(x), y, \ldots, \sigma^{M - 1}(y)]$.
  Consider the substitution
  \[
  \sigma^k(x) = \frac{k}{M} \text{ for every } 0 \leqslant k \leqslant M, \;\; \sigma^k(y) = \frac{M}{k - M} \text{ for every } 1\leqslant k \leqslant M - 1, \;\; y = 0.
  \]
  A direct computation shows that the polynomials $f_1, \ldots, \sigma^{M - 1}(f_1), f_2, \ldots, \sigma^{M - 1}(f_2)$
  vanish after this substitution, but $f$ does not.

\subsection{Non-existence of coefficient-independent effective full elimination theorem.}\label{subsec:no_full}

Let $F \subset k\{\mathbf{x}, \mathbf{u}\}$ be a finite set of difference polynomials and $h$  a positive integer.
Since $k\big[\mathbf{x}, \ldots, \sigma^h(\mathbf{x})\big]$ is Noetherian, there exists a positive integer $\ell$ such that
\begin{equation}\label{eq:full_elim}
\langle \sigma^i(F) \mid 0 \leqslant i < \infty \rangle \cap k\big[\mathbf{x}, \ldots, \sigma^h(\mathbf{x})\big] = \langle \sigma^i(F) \mid 0 \leqslant i < \ell \rangle \cap k\big[\mathbf{x}, \ldots, \sigma^h(\mathbf{x})\big].
\end{equation}
A bound on such an $\ell$ in terms of $h$, degrees and orders of $F$, and the number of 
variables
would be a natural difference counterpart of the full elimination result for differential-algebraic equations~\cite[Theorem~3]{OPV2017}.
However, the following modification of the example from Section~\ref{subsec:strong_counterexample} shows that such a bound does not exist.
We fix a positive integer $M$ and consider  system~\eqref{eq:strong_counterexample} with one extra equation
\[
  f_3 = z - y(x - 1) + 1 = 0, 
\]
where $z$ is a new unknown.
We have shown in Section~\ref{subsec:strong_counterexample} that $y(x - 1) - 1$ vanishes on every solution of~\eqref{eq:strong_counterexample} in $\mathbb{C}^{\mathbb{N}}$.
Then $z$ vanishes on every solution of $f_1 = f_2 = f_3 = 0$ in $\mathbb{C}^{\mathbb{N}}$.
Then  Hilbert's Nullstellensatz~\cite[\href{https://stacks.math.columbia.edu/tag/00FU}{Tag 00FU}]{stacks-project} combined with the Rabinowitz trick implies that there exists a positive integer $N$ such that 
\[
z^N \in  \langle \sigma^i(\{f_1, f_2, f_3\}) \mid 0 \leqslant i < \infty \rangle.
\]
On the other hand, following the argument from Section~~\ref{subsec:strong_counterexample}, we see that
\[
z^N \not\in \langle \sigma^i(\{f_1, f_2, f_3\}) \mid 0 \leqslant i < M \rangle.
\]
Thus, an integer $\ell$ such that 
\[
\langle \sigma^i(\{f_1, f_2, f_3\}) \mid 0 \leqslant i < \infty \rangle \cap \mathbb{C}[z] = \langle \sigma^i(\{f_1, f_2, f_3\}) \mid 0 \leqslant i < \ell \rangle \cap \mathbb{C}[z]
\]
must satisfy $\ell \geqslant M$. Hence there is no coefficient-independent bound for such an $\ell$.


\section{Proofs of the main results}
\label{sec:proof}

\subsection{Difference Nullstellensatz}

\begin{defin}[Inversive difference rings]
\begin{itemize}
\item[]
\item
We say that a difference ring $(A,\sigma)$ is \emph{inversive} if 
$\sigma:A \to A$ is an automorphism. 
\item
For any difference ring $(A,\sigma)$,
there is an inversive difference ring $(A^{\operatorname{inv}},\sigma)$ and 
a map of difference ring $(A,\sigma) \to (A^{\operatorname{inv}},\sigma)$
that is universal for maps from $(A,\sigma)$ to inversive difference rings (see~\citep[Proposition~2.1.7]{LevinBook}).
\item
Given a difference ring $(A,\sigma)$, the {\em ring of inversive difference 
polynomials} over $A$ in the variables, $A \{ x_1, \ldots, x_n \}^*$, 
is realized as the ordinary polynomial ring over $A$ in the 
formal variables $\sigma^j(x_i)$, for $j \in \mathbb{Z}$ and $1 \leq i \leq n$, 
with $\sigma$ extending the given endomorphism on $A$ and 
\[
  \sigma(\sigma^j(x_i)) = \sigma^{j+1}(x_i)
\] 
on the variables. 
\item If $(A,\sigma)$ is inversive, 
then so is $A \{ x_1, \ldots, x_n \}^*$.
\end{itemize}
\end{defin}

\begin{defin}
  Let $k$ be a difference field, $F \subset k\{ x_1, \ldots, x_n\}$ a finite 
  set of difference polynomials, and $h = \max \{\ord f \:|\: f \in F\}$.
  The set of $n$ tuples $\mathbf{a}_1, \ldots, \mathbf{a}_n \in k^{\ell + h}$, where
  $\mathbf{a}_i := (a_{i, 0}, \ldots, a_{i, \ell + h - 1})$, is called \emph{a partial solution of length $\ell$}
  if, for every $f \in F$ and $0 \leqslant s \leqslant \ell - 1$, the polynomial $\sigma^s(f)$ vanishes after the substitution
  \[
  \sigma^i(x_j) = a_{j, i} \text{ for every } 1 \leqslant j \leqslant n,\;\; 0 \leqslant i \leqslant \ell + h - 1.
  \]
\end{defin}

Let $K$ be an inversive difference field. Then the difference ring of sequences $K^{\mathbb{Z}}$ with respect to the shift automorphism can be endowed with a structure of a difference $K$-algebra via the embedding of difference rings $i_K \colon K \to K^{\mathbb{Z}}$ defined by 
\[
  i_K(f) = \left(\ldots, \sigma^{-1}(f), f, \sigma(f), \sigma^2(f), \ldots \right) \text{ for } f \in K.
\]
This can be similarly done  for $K^{\mathbb{N}}$.

\begin{prop}
\label{thm:DNSS}
For all uncountable algebraically closed inversive difference fields $K$ and finite sets  $F \subseteq K \{ x_1, \ldots, x_n \}$,   
 the following statements are equivalent:
\begin{itemize}
\item[1.] $F$ has a solution in $K^{\mathbb Z}$.
\item[2.] $F$ has a solution in $K^{\mathbb N}$.
\item[3.] $F$ has finite partial solutions of length $\ell$ for all
$\ell \gg 0$.
\item[4.] The ideal $[F] := \langle \{ \sigma^j(F) \mid j \in \mathbb{N}\}\rangle 
\subseteq K \{ x_1, \ldots, x_n \}$ does not contain $1$.
\item[5.] The ideal $[F]^* := \langle \{ \sigma^j(F) \mid j \in \mathbb{Z} \}\rangle 
\subseteq K \{ x_1, \ldots, x_n \}^*$ does not contain $1$.	
\item[6.] $F$ has a solution in some difference $K$-algebra.
\end{itemize}
\end{prop}

\begin{proof}
  The implications $1 \implies 2$, $2 \implies 3$, and $6 \implies 4$ are straightforward.
  
  $3 \implies4$. Assume that there exist arbitrarily long partial solutions, but $1 \in [F]$.
  Then there is an expression of the form
  \begin{equation}\label{eq:representation_one}
  1 = \sum\limits_{i = 0}^\ell \sum\limits_{f \in F} a_{i, f} \sigma^i(f),
  \end{equation}
  where $a_{i, f} \in K\{ x_1, \ldots, x_n \}$.
  Let $h = \max \{ \ord f \mid f \in F \}$.
  Consider a partial solution of $F$ of length $\ell + h + 1$ and plug it into the equality~\eqref{eq:representation_one}.
  Then the right-hand side will vanish, so we arrive at contradiction.
  
  $4 \implies 5$. Assume that $1 \in [F]^*$. We fix some representation of $1$ as an element of $[F]^*$.
  Let $N$ be the maximum number such that $\sigma^{-N}(x_i)$ occurs in the representation.
  Applying $\sigma^N$ to both sides of the representation, we obtain a representation of $1$ as an element of $[F]$.
  
  $5 \implies 6$. Let $\pi\colon K\{ x_1, \ldots, x_n \}^* \to K\{ x_1, \ldots, x_n \}^* / [F]^*$ be the canonical surjection.
  Then $\left( \pi(x_1), \ldots, \pi(x_n) \right)$ is a solution of $F$ in  $K\{ x_1, \ldots, x_n \}^* / [F]^*$.
  
  $5 \implies 1$. Let $E$ be the inversive difference subfield of $K$ generated by the coefficients of elements of $F$ over the prime subfield of $K$.
  Since $1$ does not belong to $[F]^* \cap E\{ x_1, \ldots, x_n \}^*$, there exists a maximal (not necessarily difference) ideal
  $\mathfrak{m} \subset E\{ x_1, \ldots, x_n\}^*$ containing $[F]^* \cap E\{ x_1, \ldots, x_n \}^*$.
  Then $L := E\{ x_1, \ldots, x_n\}^* / \mathfrak{m}$ is a field, and the transcendence degree of $L$ over $E$ is at most countable.
  Since $K$ is algebraically closed and has uncountable transcendence degree, there exists an embedding $\varphi\colon L \to K$ over the common subfield $E$.
  Composing $\varphi$ with the canonical surjection $E\{ x_1, \ldots, x_n\}^* \to L$, we obtain an $E$-algebra homomorphism $\psi \colon E\{ x_1, \ldots, x_n\}^* \to K$
  such that $[F]^* \subset \operatorname{Ker} \psi$.
  For every $1 \leqslant i \leqslant n$, we construct a sequence $\mathbf{a}_i := \{ a_{i, j} \}_{j \in \mathbb{Z}} \in K^{\mathbb{Z}}$ 
  by the formula \[a_{i, j} = \psi\left( \sigma^j(x_i) \right).\]
A  direct computation shows that $(\mathbf{a}_1, \ldots, \mathbf{a}_n)$ is a solution of $F$ in $K^{\mathbb{Z}}$.
\end{proof}

\subsection{Variety and two projections}\label{subsec:var_proj}

Let $k$ be a difference field and
$F = 0$
a system of difference equations, 
where $F = \{f_1,\ldots,f_N\} \subset k \{ u_1, \ldots, u_r \}$.
We set 
\[
  h_i := \max\limits_{j = 1, \ldots, N} \ord_{u_i} f_j\quad \text{and}\quad H = h_1 + \ldots + h_r  + r.
\]
For the rest of Section~\ref{sec:proof}, we fix $K$ to be an inversive uncountable algebraically closed difference field 
containing $k$.
With the system $F = 0$ of difference equations, we associate the following geometric data:
\begin{itemize}
  \item the subvariety $X$ of $\mathbb{A}^H$ defined by the polynomials $f_1, \ldots, f_N$;
  \item two projections $\pi_1, \pi_2 \colon \mathbb{A}^{H} \to \mathbb{A}^{H - r}$ defined by
\begin{align}
  &\pi_1\big( u_1, \ldots, \sigma^{h_1}(u_1), u_2, \ldots, \sigma^{h_r}(u_r) \big) := \big( u_1, \ldots, \sigma^{h_1 - 1}(u_1), u_2, \ldots, \sigma^{h_r - 1}(u_r) \big),\label{eq:pi1}\\
  &\pi_2\big( u_1, \ldots, \sigma^{h_1}(u_1), u_2, \ldots, \sigma^{h_r}(u_r) \big) := \big( \sigma(u_1), \ldots, \sigma^{h_1}(u_1), \sigma(u_2), \ldots, \sigma^{h_r}(u_r) \big).\label{eq:pi2}
\end{align}
\end{itemize}

Let $Z \subset \mathbb{A}^{H}$ be a variety defined by polynomials $g_1, \ldots, g_s \in K[\mathbb{A}^H]$.
Let $\sigma^i (Z)$, where $i \in \mathbb{Z}$,  denote the variety defined by the polynomials $g_1^{\sigma^i}, \ldots, g_s^{\sigma^i} \in K[\mathbb{A}^H]$,
where $g^{\sigma^i}$ means the result of applying $\sigma^i$ to all coefficients of $g$.
The coordinate-wise application of $\sigma^i$ defines a bijection between $Z$ and $\sigma^{i}(Z)$.

\begin{defin}\label{def:partial_solution}
 A sequence $p_1, \ldots, p_{\ell} \in \mathbb{A}^H(K)$  is 
  \emph{a partial solution} of the triple $(X, \pi_1, \pi_2)$ if
  \begin{align*}
  \begin{cases}
    \pi_1(p_{i + 1}) = \pi_2(p_i) \text{ for every } 1\leqslant i < \ell,\\
    p_i \in \sigma^{i - 1}(X)(K) \text{ for every } 1 \leqslant i \leqslant \ell.
  \end{cases}
  \end{align*}
  A two-sided infinite sequence with such a property is called \emph{a solution} of the triple $(X, \pi_1, \pi_2)$.
\end{defin}

\begin{lem}\label{lem:sol_geomertic}
   For every positive integer $\ell$, the system~$F = 0$ has 
   a partial solution of length $\ell$ if and only if the triple $(X, \pi_1, \pi_2)$ 
   has a partial solution of length $\ell$.
   
   The system~$F = 0$ has a solution in $K^{\mathbb{Z}}$ if and only if 
   the triple $(X, \pi_1, \pi_2)$ has an infinite solution.
\end{lem}

\begin{proof}
  Let $h = \max\limits_{1 \leqslant i \leqslant r} h_i$.
  Consider a partial solution $\mathbf{u}_1, \ldots, \mathbf{u}_r \in K^{\ell + h}$ of $F$, where
  $\mathbf{u}_i = (u_{i, 1}, \ldots, u_{i, \ell + h})$ for every $1 \leqslant i \leqslant r$.
  We set
  \[
  p_j := \left( u_{1, j}, \ldots, u_{1, j + h_1}, u_{2, j}, \ldots, u_{r, j + h_r} \right) \text{ for every } 1 \leqslant j \leqslant \ell.
  \]
  By the construction
  \[
  \pi_2(p_j) = \left( u_{1, j + 1}, \ldots, u_{1, j + h_1}, u_{2, j + 1}, \ldots, u_{r, j + h_r} \right) = \pi_1(p_{j + 1}),
  \]
  so $p_{j + 1} \in \pi_1^{-1} \left( \pi_2(p_j) \right)$ for every $1 \leqslant j \leqslant \ell - 1$.
  The definition of partial solution implies that $p_{j} \in \sigma^{j - 1}(X)$ for every $1 \leqslant j \leqslant \ell$.
  Hence, $p_1, \ldots, p_\ell$ is a partial solution of the triple $(X, \pi_1, \pi_2)$. 
  The above argument can be straightforwardly reversed to construct a partial solution of $F$ from a partial solution of $(X, \pi_1, \pi_2)$.
  The case of infinite solutions is completely analogous.
\end{proof}

In the introduced geometric language, we can formulate the following question equivalent to an effective difference Nullstellensatz
\begin{question}
	Let $X$ be an algebraic subvariety of $\mathbb{A}^H$ and $\pi_1, \pi_2$ be the surjective linear maps $\mathbb{A}^H \to \mathbb{A}^{H - r}$ defined by~\eqref{eq:pi1} and~\eqref{eq:pi2}.
    How long a partial solution of $(X, \pi_1, \pi_2)$ is it sufficient to find in order to 
    conclude that the triple $(X, \pi_1, \pi_2)$
    has an infinite solution?
\end{question}

Thus, in what follows, we fix a triple $(X, \pi_1, \pi_2)$, where $X$ is an algebraic subvariety of
$\mathbb{A}^H$ 
and $\pi_1, \pi_2$ are surjective linear maps $\mathbb{A}^H \to \mathbb{A}^{H - r}$ defined over the $\sigma$-constants of $K$.

\subsubsection{Trains}
The goal of this section is to generalize the notion of a solution of the triple to not necessarily zero-dimensional points.

\begin{defin}
  For $\ell$ a positive integer or $+\infty$, a sequence of irreducible subvarieties 
  $(Y_1, \ldots, Y_\ell)$ in $\mathbb{A}^H$ is said to be \emph{a train} of length $\ell$ in $X$ if 
  \[
  \begin{cases}
   \overline{\pi_1(Y_{i + 1})} = \overline{\pi_2(Y_i)} \text{ for every } 1 \leqslant i < \ell, \text{where $\overline{Y}$ denotes the Zariski closure of $Y$}, \\
    Y_i \subset \sigma^{i - 1}(X) \text{ for every } 1 \leqslant i \leqslant \ell.
  \end{cases}
  \]
\end{defin}

\begin{rem}
Let $p_1, \ldots, p_\ell \in \mathbb{A}^H$ be a partial solution of $(X, \pi_1, \pi_2)$ (see Definition~\ref{def:partial_solution}).
  Considering the singletons $\{ p_1 \}, \ldots, \{ p_\ell \}$ as irreducible zero-dimensional subvarieties of $\mathbb{A}^H$, we see that 
  $(\{ p_1 \}, \ldots, \{p_\ell\})$ is a train in $X$.
\end{rem}

\begin{lem}\label{lem:infinite_solution}
	For every train $(Y_1, \ldots, Y_\ell)$ in $X$, there exists a partial solution 
    $p_1, \ldots, p_\ell$ of $(X, \pi_1, \pi_2)$ such that, 
    for all $i$, $1 \leqslant i \leqslant \ell$, we have $p_i \in Y_i$.
\end{lem}

\begin{proof}
	We will prove the following statement by induction on $\ell$: 
    there exists a nonempty open subset $U \subset Y_\ell$
    such that, for every point $p_\ell \in U$, there exists a partial solution 
    $p_1, \ldots, p_{\ell}$ of $(X, \pi_1, \pi_2)$ such that, 
    for every $i$, $1 \leqslant i \leqslant \ell$, we have $p_i \in Y_i$.
    In the case $\ell = 1$, we can set $U = Y_1$, because every single point in $X$ is 
    a partial solution of length one.
    
    Assume that $\ell > 1$.
    Applying the inductive hypothesis to the train $(Y_1, \ldots, Y_{\ell - 1})$, 
    we obtain an open nonempty subset $U_0 \subset Y_{\ell - 1}$.
    Since $U_0$ is dense in $Y_{\ell - 1}$, $\pi_2(U_0)$ is dense in 
    $\overline{\pi_2(Y_{\ell - 1})} = \overline{\pi_1(Y_\ell)}$.
    Since $\pi_1(Y_\ell)$ is a constructible dense subset in $\overline{\pi_1(Y_\ell)}$, 
    $\pi_2(U_0) \cap \pi_1(Y_\ell)$ is also dense constructible in $\overline{\pi_1(Y_\ell)}$.
    Let $U_1 \subset \pi_2(U_0) \cap \pi_1(Y_\ell)$ be an open dense subset of $\overline{\pi_1(Y_\ell)}$.
    Then $U_2 := Y_\ell \cap \pi_1^{-1}(U_1)$ is nonempty open in $Y_{\ell}$.
    We claim that every point $p_\ell \in U_2$ can be extended to a partial solution 
    $p_1, \ldots, p_{\ell}$ such that $p_i \in Y_i$.
    By the definition of $U_2$, $\pi_1(p_\ell) \in \pi_2(U_0)$, 
    so there exists $p_{\ell - 1} \in U_0$ such that $\pi_2(p_{\ell - 1}) = \pi_1(p_\ell)$.
    By the inductive hypothesis, $p_{\ell - 1}$ can be further extended to a partial solution.
\end{proof}

\begin{cor}\label{cor:infinite_solution}
  If there is an infinite train in $X$, then there is a solution for the triple $(X, \pi_1, \pi_2)$.
\end{cor}

\begin{proof}
  Since there is an infinite train, there are arbitrarily long finite trains.
  Due to Lemma~\ref{lem:infinite_solution}, there are arbitrarily long finite partial solutions of $(X, \pi_1, \pi_2)$.
  Lemma~\ref{lem:sol_geomertic} implies that there are arbitrarily long finite partial solutions
  of the corresponding system~$F$.
  Hence, due to Proposition~\ref{thm:DNSS}, there is a solution of~$F$ in $K^{\mathbb{Z}}$.
  Lemma~\ref{lem:sol_geomertic} implies that there exists an infinite solution of the triple $(X, \pi_1, \pi_2)$.
\end{proof}

\begin{defin}[Train operations]
\begin{itemize}
\item[]
\item
  For two trains $Y$ and $Y'$ of the same length, the inclusion $Y \subset Y'$ is 
  understood as a component-wise containment.
\item
  For a train $Y$ in $X$ and $i \in\mathbb{Z}$, $\sigma^{i}(Y)$ is the result of 
  the component-wise application of $\sigma^i$ to $Y$, and, since $\pi_1$ and $\pi_2$ are defined 
  over the constants, $\sigma^i(Y)$ is a train in $\sigma^i(X)$.
  \end{itemize}
\end{defin}

\begin{rem}
  Since the component-wise union of any chain of trains of the same length is again a train of this length, 
  trains of fixed length satisfy Zorn's lemma with respect to inclusion.
  Hence,  maximal trains of a fixed length are well-defined.
\end{rem}

\subsubsection{The number of maximal trains}
 
Our next Lemma~\ref{lem:component_fiber_product} appears to be part of the 
 folklore, but for want of a written reference, we offer a proof here.

\begin{lem}\label{lem:component_fiber_product}
  Let $\varphi_X\colon X \to Z$ and $\varphi_Y \colon Y \to Z$ be dominant morphisms of 
  affine varieties over an algebraically closed field.
  Assume that $X$ and $Y$ are irreducible.
  Consider the fibered product $X \times_Z Y$ of $\varphi_X$ and $\varphi_Y$,
  considered as a variety, and denote the natural morphisms 
  to $X$ and $Y$ by $\pi_X$ and $\pi_Y$, respectively.
  Then there exists an irreducible component $V \subset X\times_Z Y$ such that the restrictions of 
  both $\pi_X$ and $\pi_Y$ to $V$ are dominant.
\end{lem}

\begin{proof}
  Denote the algebras of regular functions on $X$, $Y$, and $Z$ by $A$, $B$, and $C$, respectively.
  Since $X$, $Y$, and $Z$ are irreducible ($Z$ is irreducible as an image of an irreducible variety 
  under a dominant morphism), these algebras are domains.
  We denote the fields of fractions of $A$, $B$, and $C$ by $E$, $F$, and $L$, respectively.
  The dominant maps $\varphi_X$  and $\varphi_Y$ give rise to injective homomorphisms 
  $\varphi_X^{\#} \colon C \to A$ and $\varphi_Y^{\#} \colon C \to B$.
  These homomorphisms equip $A$ and $B$ with a $C$-algebra structure.
  Then, the algebra of regular functions on $X \times_Z Y$, as a scheme, is $A \otimes_C B$ 
  (see ~\cite[\href{https://stacks.math.columbia.edu/tag/01I4}{Tag 01I4}]{stacks-project}).
  
  Let $\mathfrak{p}$ be any prime ideal in $E \otimes_L F$.
  Let $D := (E \otimes_L F) / \mathfrak{p}$ and $\pi\colon E \otimes_L F \to D$ be the canonical projection.
  Consider the natural homomorphism $i \colon A \otimes_C B \to E \otimes_L F$.
  Since $1 \in i(A \otimes_C B)$, the composition $\pi \circ i$ is a nonzero homomorphism.
  Consider the natural embeddings $i_A \colon A \to A \otimes_C B$ and $i_B \colon B \to A\otimes_C B$.
  We will show that the compositions $\pi\circ i \circ i_A \colon A \to D$ and 
  $\pi\circ i \circ i_B \colon B \to D$ are injective.
  Introducing the natural embeddings $i_E \colon E \to E\otimes_L F$ and $j_A \colon A \to E$, we can rewrite
  \[
  \pi\circ i \circ i_A = \pi \circ i_E \circ j_A.
  \]
  The homomorphisms $i_E$ and $j_A$ are injective.
  The restriction of $\pi$ to $i_E(E)$ is also injective, since $E$ is a field.
  Hence, the whole composition $\pi \circ i_E \circ j_A$ is injective. 
  The argument for $\pi\circ i \circ i_B$ is analogous.
  
  Thus, we have an irreducible subvariety of $X \times_Z Y$, and hence of the 
 variety $(X \times_Z Y)_{\operatorname{red}}$~\cite[\href{https://stacks.math.columbia.edu/tag/0356}{Tag 0356}]{stacks-project}, 
  defined by the ideal 
  $\operatorname{Ker} (\pi \circ i)$ that projects dominantly on both $X$ and $Y$.
  Hence, the component containing this subvariety also projects dominantly on $X$ and $Y$. 
\end{proof}

\begin{defin}[Marked trains]
  Let $X_1 \cup X_2 \cup \ldots \cup X_s$ be the decomposition of $X$ into irreducible components.
  \begin{itemize}
  \item
  A pair $(Y, \mathbf{c})$, where $Y = (Y_1, \ldots, Y_\ell)$ is a train in $X$ and 
  $\mathbf{c} = (c_1, \ldots, c_\ell) \in \{1, \ldots, s\}^\ell$,
  is called a \emph{marked train} of length $\ell$ and \emph{signature} $\mathbf{c}$ if 
  $Y_i \subset \sigma^{i - 1}(X_{c_i})$ for every $1 \leqslant i \leqslant \ell$.
\item Every train has at least one signature (maybe several), so that it becomes a marked train.
\item Analogously to trains, we define a notion of a maximal train of given length $\ell$ and signature $\mathbf{c}$.
\end{itemize}
\end{defin}
Let $\mathbf{c} = (c_1, \ldots, c_\ell) \in \{1, \ldots, s\}^\ell$ be a tuple.
Consider 
\[
  X^{\mathbf{c}} := X_{c_1} \times \sigma(X_{c_2}) \times \ldots \times \sigma^{\ell - 1}(X_{c_\ell}) \subset (\mathbb{A}^H)^\ell.
\]
We denote the projections 
$(\mathbb{A}^H)^{\ell} \to \mathbb{A}^H$ onto the components by $\psi_{\ell, 1}, \ldots, \psi_{\ell, \ell}$, respectively.
We introduce
\begin{equation}\label{eq:defW}
  W_{\mathbf{c}} := \left\{ p \in X^{\mathbf{c}}\: \big|\: \pi_2\left( \psi_{\ell, i}(p) \right) = \pi_1\left( \psi_{\ell, i + 1}(p) \right) \text{ for all } i,\,1 \leqslant i < \ell \right\}.
\end{equation}
The restrictions of $\psi_{\ell, 1}, \ldots, \psi_{\ell, \ell}$ to $W_{\mathbf{c}}$ will be denoted by the same symbols.

\begin{lem}\label{lem:variety_to_train}
  For every irreducible subvariety $Z \subset W_{\mathbf{c}}$, 
  \[
    \left( \overline{\psi_{\ell, 1}(Z)}, \ldots, \overline{\psi_{\ell, \ell}(Z)} \right)
  \]
  is a marked train of signature $\mathbf{c}$.
\end{lem}

\begin{proof}
  For every $i$, $1 \leqslant i \leqslant \ell$, 
  \[
    Y_i := \overline{\psi_{\ell, i}(Z)} \subset \overline{\psi_{\ell, i}(W_{\mathbf{c}})} \subset \sigma^{i - 1}(X_{c_i}).
  \]
  Moreover, since $Z$ is irreducible,  $\overline{\psi_{\ell, i}(Z)}$ is also irreducible.
  Fix some $i$, $1 \leqslant i < \ell$.
  We will show that $\overline{\pi_2(Y_i)} = \overline{\pi_1(Y_{i + 1})}$.
  We can write $\overline{\pi_2(Y_i)}$ as $\overline{\pi_2(\psi_{\ell, i}(Z))}$.
  By~\eqref{eq:defW}, the latter is equal to $\overline{\pi_1(\psi_{\ell, i + 1}(Z))}$, which is the same as $\overline{\pi_1(Y_{i + 1})}$.
\end{proof}

\begin{lem}\label{lem:train_to_variety}
  For every marked train $(Y_1, \ldots, Y_\ell)$ of signature $\mathbf{c}$ in $X$,
  there exists an irreducible subvariety $Y \subset W_{\mathbf{c}}$ such that, 
  for every $i$, $1 \leqslant i \leqslant \ell$, 
  we have $Y_i = \overline{\psi_{\ell, i}(Y)}$.
\end{lem}

\begin{proof}
  We will prove the lemma by induction on $\ell$.
  For $\ell = 1$, $\mathbf{c} = (c_1)$, $W_{\mathbf{c}} = X_{c_1}$, and we can set $Y = Y_1$.
  
  Let $\ell > 1$. Apply the inductive hypothesis to the train $(Y_1, \ldots, Y_{\ell - 1})$ of signature $\mathbf{c}' := (c_1, \ldots, c_{\ell - 1})$ 
  and obtain an irreducible subvariety $Y' \subset W_{\mathbf{c}'} \subset (\mathbb{A}^H)^{\ell - 1}$.
  Then there is a natural embedding of $Y' \times Y_\ell$ into $(\mathbb{A}^H)^\ell$.
  Denote $(Y' \times Y_\ell) \cap W_{\mathbf{c}}$ by $W$.
  Since $Y'$ is already contained in $W_{\mathbf{c}'}$,
  \begin{equation}\label{eq:reprW}
  W = \{ p \in Y' \times Y_\ell \:|\: \pi_2\left( \psi_{\ell, \ell - 1}(p) \right) = \pi_1\left( \psi_{\ell, \ell}(p) \right) \}.
  \end{equation}
Let $\psi = (\psi_{\ell,1}, \psi_{\ell,2}, \ldots, \psi_{ \ell,\ell - 1}) \colon (\mathbb{A}^H)^\ell \to (\mathbb{A}^H)^{\ell - 1}$ and 
  \begin{equation}\label{eq:reprZ}
  Z := \overline{\pi_2\left( \psi_{\ell - 1, \ell - 1}(Y') \right)} = \overline{\pi_1(Y_\ell)}.
  \end{equation}
  Then equality~\eqref{eq:reprW} implies (see~\cite[Ex.~2.26]{Harris}) that $W$ together with the
  morphisms
  $\psi\colon W \to Y'$ and $\psi_{\ell, \ell}\colon W \to Y_\ell$
  is the fibered product of the morphisms $\pi_2\circ \psi_{\ell - 1, \ell - 1} \colon Y' \to Z$ and $\pi_1 \colon Y_\ell \to Z$.
  Equality~\eqref{eq:reprZ} implies that both of these morphisms are dominant.
  
  Due to Lemma~\ref{lem:component_fiber_product}, there exists an irreducible subset $Y \subset W$ 
  such that $\psi\colon Y\to Y'$ and $\psi_{\ell, \ell} \colon Y \to Y_\ell$ are dominant.
  For every $i$, $1 \leqslant i < \ell$, since $\psi_{\ell, i} = \psi_{\ell - 1, i} \circ \psi$, 
  the restriction $\psi_{\ell, i} \colon Y \to Y_i$ is dominant  as a composition of two dominant morphisms.
\end{proof}

\begin{lem}\label{lem:number_trains}
  Let the degree of $X_i$ be $D_i$ (see Definition~\ref{def:deg}), and fix a tuple $\mathbf{c} = (c_1, \ldots, c_\ell) \in \{1, \ldots, s\}^\ell$.
  The number of maximal trains of signature $\mathbf{c}$ in $X$ does not exceed $D_{c_1} \cdot D_{c_2} \cdot \ldots\cdot D_{c_\ell}$.
\end{lem}

\begin{proof}
  Since $W_{\mathbf{c}}$ is the intersection of $X^{\mathbf{c}}$ with a linear subspace, 
  \begin{equation}\label{eq:prod}
    \deg W_{\mathbf{c}}\leqslant\deg X^{\mathbf{c}}= \deg X_{c_1}\cdot \deg \sigma(X_{c_2})\cdot \ldots\cdot \deg \sigma^{\ell-1}(X_{c_\ell}).
  \end{equation}
  Since application of $\sigma$ to a variety does not change the degree, the product in~\eqref{eq:prod} does not exceed $D_{c_1}\cdot\ldots\cdot D_{c_\ell}$.
  Hence, the number of components of $W_{\mathbf{c}}$ does not exceed $D_{c_1}\cdot\ldots\cdot D_{c_\ell}$.
  
  Let $(Y_1, \ldots, Y_\ell)$ be a maximal train in $X$ of signature $\mathbf{c}$.
  Lemma~\ref{lem:train_to_variety} implies that there exists an irreducible subvariety $Y \subset W_{\mathbf{c}}$
  such that for every $i$, $1 \leqslant i \leqslant \ell$, we have $Y_i = \overline{\psi_{\ell, i}(Y)}$.
  Let $C$ be an irreducible component of $W_{\mathbf{c}}$ containing $Y$.
  Lemma~\ref{lem:variety_to_train} implies that 
  \[
    \left(\overline{\psi_{\ell, 1}(C)}, \ldots, \overline{\psi_{\ell, \ell}(C)}\right)
  \]
  is also a train of signature $\mathbf{c}$. Moreover, since $C \supset Y$, this train contains $(Y_1, \ldots, Y_\ell)$.
  The maximality of the latter implies that these trains are equal.
  Hence, $Y$ could be chosen to be an irreducible component of $W_{\mathbf{c}}$.
  Thus, we obtain an injective map from the set of maximal trains of signature $\mathbf{c}$ to the set of all irreducible component of $W_{\mathbf{c}}$.
  Hence, the number of maximal trains also does not exceed $D_{c_1}\cdot\ldots\cdot D_{c_\ell}$.
\end{proof}

\begin{cor}\label{cor:number_trains}
  The number of maximal trains in $X$ of length $\ell$ does not exceed $(\deg X)^\ell$.
\end{cor}

\begin{proof}
  Since every maximal train can be considered as a marked maximal train, the number of maximal trains of length $\ell$
  in $X$ does not exceed the sum of products $D_{c_1}\cdot\ldots\cdot D_{c_\ell}$ over all tuples $\mathbf{c}$ of length $\ell$.
  This sum is equal to
  \[
  \sum\limits_{c_1 = 1}^s \sum\limits_{c_2 = 1}^s \ldots \sum\limits_{c_\ell = 1}^s \prod\limits_{i = 1}^\ell D_{c_i} = (D_1 + \ldots + D_s)^\ell = D^\ell.\qedhere
  \]
\end{proof}

\subsubsection{A bound for trains}\label{subsec:bound_for_trains}

\begin{defin}
For a train $Y = (Y_1, \ldots, Y_\ell)$ in $X$, we introduce the codimension of $Y$ as
\[
  \codim Y := \dim X - \min\limits_{1 \leqslant i \leqslant \ell} \dim Y_i.
\]
\end{defin}

\begin{defin}
  We call a train $Y = (Y_1, \ldots, Y_\ell)$ in $X$ \emph{skew-cyclic} if $\ell > 1$ and $Y_\ell = \sigma^{\ell - 1}(Y_1)$.
\end{defin}

\begin{lem}\label{lem:cycling}
  If there exists a skew-cyclic train in $X$ of codimension $d$, then there exists an infinite train in $X$ of codimension $d$.
\end{lem}

\begin{proof}
 Let $(Y_1, \ldots, Y_\ell)$ be a skew-cyclic train in $X$ of codimension $d$. 
 Then we can construct an infinite train of codimension $d$ as follows:
   \[
     \left(Y_1, Y_2, \ldots, Y_{\ell - 1}, \sigma^{\ell - 1}(Y_1), \sigma^{\ell - 1}(Y_2), \ldots, \sigma^{\ell - 1}(Y_{\ell - 1}), \sigma^{2\ell - 1}(Y_1),\ldots\right).\qedhere
   \]
\end{proof}

\begin{defin}
  We define $B'(d, D)$ to be the smallest natural number $N$ such that, for every triple $(X, \pi_1, \pi_2)$
  such that the $\deg X\leqslant D$, the
  existence of a train of length $N$ and codimension at most $d$ in $X$ implies the existence of a skew-cyclic train in $X$ of length at most $N$ and codimension at most $d$, or $\infty$ if such $N$ does not exist.
\end{defin}

The following statement implies that $B'(d, D)$ is finite for all $d \in \mathbb{Z}_{\geqslant 0}$ and $D \in \mathbb{Z}_{> 0}$
and gives an upper bound for $B'(d, D)$.

\begin{prop}\label{prop:bound_general}  For all $D \in \mathbb{Z}_{> 0}$,
  \begin{enumerate}
    \item $B'(0, D) \leqslant D + 1$ and
    \item for every $d \in \mathbb{Z}_{\geqslant 0}$, $B'(d + 1, D) \leqslant B'(d, D) + D^{B'(d, D)}$.
  \end{enumerate}
\end{prop}

\begin{proof}
  Throughout the proof, we will use the observation that the existence of a skew-cyclic train in $\sigma^i(X)$ for some $i \in \mathbb{Z}$ implies (via component-wise application of $\sigma^{-i}$) the existence of a skew-cyclic train of the same codimension in $X$.

  We prove the first statement of the proposition.
  Consider a train $(Y_1, \ldots, Y_{D + 1})$ of codimension zero and length $D + 1$.
  Since, for every $i$, $1 \leqslant i \leqslant D + 1$, we have $\dim Y_i = \dim X$, then every $\sigma^{-i + 1}(Y_i)$ is an irreducible component of $X$.
  The number of components of $X$ does not exceed $D$, so some of the $\sigma^{-i + 1}(Y_i)$'s coincide.
  If $\sigma^{-i + 1}(Y_i) = \sigma^{-j + 1}(Y_j)$ for some $i < j$, then $Y_j = \sigma^{j - i}(Y_i)$, 
  so $(Y_i, \ldots, Y_j)$ is a skew-cyclic train of codimension zero.
  
  We prove the second statement of the proposition.
  Consider a train $(Y_1, \ldots, Y_B)$ of codimension at most $d + 1$ and length \[B := B'(d, D) + D^{B'(d, D)}.\]
  We introduce $N := D^{B'(d, D)} + 1$ trains $Z^{(1)}, \ldots, Z^{(N)}$ of length $\ell := B'(d, D)$
  in $X, \sigma(X), \ldots, \sigma^{N - 1}(X)$, respectively,
  such that, for all $i$, $1 \leqslant i \leqslant N$, we have\vspace{-0.05in}
  \[
  Z^{(i)} = \big(Z_1^{(i)}, \ldots, Z_{\ell}^{(i)}\big) := (Y_{i}, \ldots, Y_{i + \ell - 1}).  \vspace{-0.05in}
  \]
  For every $i$, $1 \leqslant i \leqslant N$, consider a maximal train $\widetilde{Z}^{(i)} = \big(\widetilde{Z}_1^{(i)}, \ldots, \widetilde{Z}_\ell^{(i)}\big)$ 
  of length $\ell$ in $\sigma^{i - 1}(X)$ containing $Z^{(i)}$.
  Then $\sigma^{-i + 1}(\widetilde{Z}^{(i)})$ is a maximal train of length $\ell$ in $X$.
  If there exists $i$ such that $\operatorname{codim} \widetilde{Z}^{(i)} \leqslant d$, then there is a skew-cyclic train of length at most $B'(d, D)$ and codimension at most $d$
  due to the definition of $B'(d, D)$.
  Otherwise, $\operatorname{codim} \widetilde{Z}^{(i)} = d + 1$ for every $1 \leqslant i \leqslant N$.
  
  Corollary~\ref{cor:number_trains} implies that there are at most $D^{\ell} = N - 1$ maximal trains of length $\ell$ in $X$.
  Hence, there are $a$ and $b$, $1 \leqslant a < b \leqslant N$, such that 
  \[
  \sigma^{-a + 1}(\widetilde{Z}^{(a)}) = \sigma^{-b + 1}(\widetilde{Z}^{(b)}).
  \]
  Since $\codim \widetilde{Z}^{(a)} = \codim \widetilde{Z}^{(b)} = d + 1$, 
  there exists $j$, $1 \leqslant j \leqslant \ell$, such that 
  \[
  \dim \widetilde{Z}^{(a)}_j = \dim \widetilde{Z}^{(b)}_j = \dim X - (d+1).
  \]
  Hence, since both $\widetilde{Z}^{(a)}_j$ and $Z_j^{(a)}$ are irreducible, $\dim Z_j^{(a)} \geqslant \dim X - (d+1)$ and $Z_j^{(a)} \subset \widetilde{Z}^{(a)}_j$, they are equal.
  Analogously, $\widetilde{Z}^{(b)}_j=Z_j^{(b)}$. 
  Therefore,
  \begin{align*}
  \sigma^{-a + 1}(Y_{a + j - 1}) &= \sigma^{-a + 1}(Z_j^{(a)}) = \sigma^{-a + 1}(\widetilde{Z}^{(a)}_j) \\&= \sigma^{-b + 1}(\widetilde{Z}^{(b)}_j) = \sigma^{-b + 1}(Z_j^{(b)}) = \sigma^{-b + 1}(Y_{b + j - 1}).
 \end{align*}
  Hence, \[Y_{b + j - 1} = \sigma^{b - a}(Y_{a + j - 1}),\] 
  so, $(Y_{a + j - 1}, Y_{a + j}, \ldots, Y_{b + j - 1})$ is a skew-cyclic train of codimension at most $d + 1$.
\end{proof}

\begin{prop}\label{prop:bound1}
  $B'(1, D) \leqslant \frac{D^3}{6} + \frac{D^2}{2} + \frac{4D}{3} + 1$ for every $D \geqslant 1$.
\end{prop}

\begin{proof}
  Let $\deg X\leqslant D$.
  Assume that there is no skew-cyclic train of codimension at most one in $X$.
  Let \[X = X_1 \cup X_2 \cup\ldots \cup X_s\] be the irreducible decomposition of $X$ and $D_i:=\deg X_i$.
  We construct a directed graph (with loops and multiple edges) $G$ with vertices numbered from $1$ to $s$ as follows.
  For every maximal train among the marked trains of signature $(i, j)$ in $X$, we draw an edge from $i$ to $j$ (the number of such trains is finite by Lemma~\ref{lem:number_trains}).
  The codimension of an edge is defined to be the codimension of the corresponding train.
  \begin{enumerate}[{Case} 1{:}\ ]
  \item
  \emph{there is a directed cycle $(c_1, \ldots, c_\ell, c_1)$ consisting of edges of codimension zero (since the graph has $s$ vertices, there would be such a cycle with $\ell \leqslant s$)}.
  Then there is a
  skew-cyclic train 
  \[
  (X_{c_1}, \sigma(X_{c_2}), \ldots, \sigma^{\ell - 1}(X_{c_\ell}), \sigma^{\ell}(X_{c_1})),
  \] 
 of codimension zero and length at most $s + 1 \leqslant D + 1$.
  \item
  \emph{there is no such a directed cycle in $G$.}
  Therefore, we can reenumerate the components in such a way that $i < j$ for every codimension zero edge $(i, j)$.
  Consider a train $Y = (Y_1, \ldots, Y_\ell)$ in $X$ of length 
  \begin{equation}\label{eq:ell}
    \ell := \frac{D^3}{6} + \frac{D^2}{2} + \frac{4D}{3} + 1
  \end{equation}
  and codimension one.
  The train $Y$ can be considered as a marked train with respect to a signature $\mathbf{c} = (c_1, \ldots, c_\ell)$.
  For every $i$, $1 \leqslant i < \ell$, we choose a maximal marked train $T_i$ of signature $(c_i, c_{i + 1})$
   in $X$ containing $\left( \sigma^{-i + 1}(Y_i), \sigma^{-i + 1}(Y_{i + 1}) \right)$
  and let $e_i$ be the edge  in $G$ corresponding to $T_i$.
  Note that \[(e_1, \ldots, e_{\ell - 1})\] is a path in $G$.
  \begin{enumerate}[{Case 2}a:\ ]
  \item
  \emph{some edge $e$ corresponding to a maximal train, denoted $(Z_1, Z_2)$, of codimension one occurs twice in this path, 
  so $e = e_i = e_j$ for some $1 \leqslant i < j < \ell$.}
  Without loss of generality, we may assume that \[\dim Z_1 = \dim X - 1.\]
  Since $\dim Y_i$ and $\dim Y_j$ are both at least $\dim X - 1$ and $(Z_1,Z_2)$ is maximal, we conclude that \[Z_1 = \sigma^{-i + 1}(Y_i) = \sigma^{-j + 1}(Y_j).\]
  Hence, $(\sigma^{-i + 1}(Y_i), \sigma^{-i + 1}(Y_{i + 1}), \ldots, \sigma^{-i + 1}(Y_j))$ is a skew-cyclic train in $X$ of length at most $\ell$ and codimension at most one.
  \item  
\emph{every edge of codimension one occurs in the path $(e_1, \ldots, e_{\ell - 1})$ at most once.}
  Until the end of the proof, we fix the path $(e_1, \ldots, e_{\ell - 1})$, and all of the quantities below are computed for this path.
  For an edge $e = (i, j)$, we introduce the weight $w(i, j) := i - j$.
  Let 
 \begin{gather*}
  N_{+} := \left\lvert \left\{ i \mid 1 \leqslant i < \ell,\; w(e_i) \geqslant 0 \right\} \right\rvert, \\
  N_{-} := \left\lvert \left\{ i \mid 1 \leqslant i < \ell,\; w(e_i) < 0 \right\} \right\rvert,\\
  W_{+} := \sum\limits_{i = 1}^{\ell - 1} \max\{0, w(e_i)\},\quad
  W_{-} := \sum\limits_{i = 1}^{\ell - 1} \min\{ 0, w(e_i) \}.
  \end{gather*}
By the above reenumeration, all edges with positive weight are of codimension at least one. Therefore, $N_{+}$ does not exceed the number of maximal marked trains with signatures of
  the form $(i, j)$ with $i \geqslant j$. Hence, due to Lemma~\ref{lem:number_trains}, we obtain
  \[
  N_{+} \leqslant \sum\limits_{1 \leqslant j \leqslant i \leqslant s} D_iD_j.
  \]
  Since the sum of weights along any path between vertices $a$ and $b$ is equal to $a - b$  and the vertices in $G$ are numbered by the integers from $1$ to $s$, 
  \[
  W_{+} + W_{-} \geqslant -s + 1.
  \]
  Combining this inequality with the fact that $N_{-} \leqslant -W_{-}$, we obtain 
  \[N_{-} \leqslant W_{+} + s - 1.\]
  Due to Lemma~\ref{lem:number_trains}, 
  \[W_{+}\leqslant\sum\limits_{1 \leqslant j < i \leqslant s} (i - j) D_iD_j.\vspace{-0.1in}\]
  Thus,
  \begin{align}\begin{split}\label{eq:d1}
  \ell = N_{+} + N_{-} + 1 \leqslant \sum\limits_{1 \leqslant j \leqslant i \leqslant s} D_iD_j + \sum\limits_{1 \leqslant j < i \leqslant s} (i - j) D_iD_j + s \\ \leqslant D^2 +\sum\limits_{1 \leqslant j < i \leqslant s} (i - j - 1) D_iD_j + s.
  \end{split}
  \end{align}
  For every integer $q \geqslant 1$, we introduce a function
  \[
  f_q(z_1, \ldots, z_q) := \sum\limits_{1\leqslant j < i\leqslant q} (i - j - 1) z_iz_j + q.
  \]
  We claim that, for every positive integer $M$, the set
  \[
  \{ f_q(z_1, \ldots, z_q) \mid q, z_1, \ldots, z_q \in \mathbb{Z}_{\geqslant 1},\; z_1 + \ldots + z_q = M\}
  \]
  reaches its maximum at $q = M$ and $z_1 = \ldots = z_q = 1$.
  
  To prove the claim, consider any integer $p\geqslant 1$ and a  tuple of positive integers $(w_1, \ldots, w_p)$. 
  Let $r \leqslant p$ be an integer
  such that $w_r \neq 1$.
  We have
  \begin{align}
  \begin{split}\label{eq:p_plus_1}
  &f_{p + 1}(w_1, \ldots, w_{r - 1}, w_r - 1, 1, w_{r + 1}, \ldots, w_p) = \\ &\quad\quad\quad\sum\limits_{j < i < r} (i - j - 1)w_iw_j + \sum\limits_{r < j < i} (i - j - 1)w_iw_j  \\ &\quad\quad\quad+ \sum\limits_{j < r}w_j((r - j - 1)w_r + 1) \\&+ \sum\limits_{r < i} w_i((i - r - 1)w_r + w_r - 1) 
  + (p + 1)
  \end{split}
  \end{align}
  and 
  \begin{align}
  \begin{split}\label{eq:p}
  &f_{p}(w_1, \ldots, w_r, \ldots, w_p) = \\ &\quad\sum\limits_{j < i < r} (i - j - 1)w_iw_j + \sum\limits_{r < j < i} (i - j - 1)w_iw_j  \\ &\quad+ \sum\limits_{j < r < i} (i - j - 1) w_i w_j + \sum\limits_{j < r}(r - j - 1)w_jw_r\\
  &\quad+ \sum\limits_{r < i} (i - r - 1)w_rw_i + p.
  \end{split}
  \end{align}
  Comparing~\eqref{eq:p_plus_1} and~\eqref{eq:p} term by term, we see that 
  \[
    f_{p + 1}(w_1, \ldots, w_{r - 1}, w_r - 1, 1, w_{r + 1}, \ldots, w_p) > f_{p}(w_1, \ldots, w_r, \ldots, w_p). 
  \]
  Hence, the claim is proved.
  Let $\widehat{D} := D_1 + \ldots + D_s$.
  Combining the claim with~\eqref{eq:d1}, we obtain
  \begin{align*}
  \ell &\leqslant D^2 + f_{\widehat{D}}(D_1, \ldots, D_s) \leqslant D^2 + f_{\widehat{D}}(1, \ldots, 1)   \\ &\leqslant D^2 + f_{D}(1, \ldots, 1)= D^2 + \sum\limits_{1 \leqslant j < i \leqslant D} (i - j - 1) + D \\&= D + D^2 + \sum\limits_{i = 1}^{D - 2} i(D - 1 - i) = \frac{D^3}{6} + \frac{D^2}{2} + \frac{4D}{3}
  \end{align*}
  and arrive at the contradiction with the definition~\eqref{eq:ell} of $\ell$.\qedhere
  \end{enumerate}
  \end{enumerate}
\end{proof}

Propositions~\ref{prop:bound_general} and~\ref{prop:bound1} imply

\begin{cor}
  For all $d \in \mathbb{Z}_{\geqslant 0}$ and $D \in \mathbb{Z}_{> 0}$, $B'(d, D) \leqslant B(d, D)$.
\end{cor}

\subsection{Proof of  effective Nullstellensatz}
\vspace{-0.05in}
\begin{proof}[Proof of  Theorem~\ref{th:nullstellensatz}]
  The $\implies$  implication is straightforward, we will prove $\impliedby$.
  We will use the notation introduced in Section~\ref{subsec:var_proj}.
  The fact that the system consisting of $0$-th, $\ldots$, $B(d, D) - 1$-th transforms of $F = 0$ considered as a polynomial system
  is consistent implies that $F = 0$ has a partial solution of length $B(d, D) \geqslant B'(d, D)$.
  Lemma~\ref{lem:sol_geomertic} implies that there is a partial solution of the triple $(X, \pi_1, \pi_2)$ of length $B'(d, D)$.
  This partial solution is a train in $X$ of codimension $\dim X = d$ and length $B'(d, D)$.
  The definition of $B'(d, D)$  and Lemma~\ref{lem:cycling} imply that there exists an infinite train in $X$. 
  Then Lemma~\ref{lem:sol_geomertic} and Corollary~\ref{cor:infinite_solution} imply that 
  the system~$F = 0$ has a solution in some difference ring extending $k$.
\end{proof}

\begin{proof}[Proof of Corollary~\ref{cor:main}]
 The $\implies$  implication is straightforward, we will prove $\impliedby$.
  We will use the notation introduced in Section~\ref{subsec:var_proj}.
  If $k = \mathbb{C}$, then $K$ can be chosen to be $\mathbb{C}$, too.
  A solution of the system $\{\sigma^i(F)=0\mid 0\leqslant i <B(d, D)\}$ yields a partial solution of $ F = 0$ of length $B(d, D)$.
  Analogously to the proof of Theorem~\ref{th:nullstellensatz}, we have that the system~$F = 0$ is consistent.
  Then Proposition~\ref{thm:DNSS} implies that $F  = 0$ has a solution in $\mathbb{C}^\mathbb{Z}$.
\end{proof}

\subsection{Proof of effective elimination}
\vspace{-0.05in}
\begin{proof}[Proof of  Theorem~\ref{thm:main2}]
  The $\impliedby$  implication is straightforward, we will prove $\implies$.
  Let  $E_0 \supset E = \operatorname{Frac(k\{\mathbf{x}\})}$ be any difference field extension such that $E_0$ is algebraically closed and 
  has uncountable transcendence degree over
  the prime subfield.
  Since the difference ideal generated by $F$ in $k\{ \mathbf{x}, \mathbf{u} \}$ contains a nonzero 
  polynomial depending only on $\mathbf{x}$ and their transforms, the difference ideal generated by $F$
  in $E_0\{ \mathbf{u}\}$ contains $1$. So, the system does not have a solution in $E_0^{\mathbb{Z}}$.
  Theorem~\ref{th:nullstellensatz} implies that the system $F = 0$ does not have a partial solution in $E_0$ of length $B(d, D)$.
  Hence, the ideal generated by $B(d, D)$ transforms of $F$ contains $1$.
  Since the ideal is defined over $E$,  there is an expression of $1$ over $E$ of the form 
  \[
  1 = \sum\limits_{i = 0}^{B(d, D) - 1}  \sum\limits_{j = 1}^N c_{i, j}\sigma^i(f_j),
  \]
  where $c_{i, j} \in E\{\mathbf{u}\}$.
  Multiplying both sides of the above equality  by the product of the denominators of $c_{i, j}$'s, we obtain an expression
  of a nonzero polynomial from $k\{ \mathbf{x}\}$ as a $k\{ \mathbf{x}, \mathbf{u}\}$-linear combination of $B(d, D)$ transforms of $F$.
\end{proof}

\section{Difference Nullstellensatz over small fields}
\label{sec:smallfields}
An hypothesis of our Proposition~\ref{thm:DNSS}
is that the field $K$ is uncountable.     
In practice, this 
is a harmless assumption as 
one might take that field to be
$\mathbb{C}$.  However, this 
result may be conceptually 
unsatisfactory, and one might 
wish to find solutions to 
difference equations in sequences
taken from a small field, such as the field of algebraic numbers.

With the next proposition, we show how to
weaken the uncountability 
hypothesis by appealing to 
a more refined equivalent
condition to the consistency of 
a system of difference equations 
coming from our work towards the 
effective Nullstellensatz and 
a theorem of Hrushovski on the 
limit theory of the Frobenius 
automorphisms~\cite{HrFrob}.  Our 
invocation of Hrushovski's theorem
is essentially contained in 
Fakhruddin's proof
of the density of periodic
points for polarized 
algebraic dynamical systems 
in~\cite{Fakh}.  
For our purposes a 
slightly weaker result due to 
Varshavsky~\cite{Varshavsky}
suffices.

\begin{thm}
\label{thm:DNSS-small}
For all algebraically closed inversive difference fields $K$ (\emph{without
any restriction on the cardinality}) and finite sets  $F \subseteq K \{ x_1, \ldots, x_n \}$,   
 the following statements are equivalent:
\begin{enumerate}[1{.}]
\item $F$ has a solution in $K^{\mathbb Z}$.
\item $F$ has a solution in $K^{\mathbb N}$.
\item\label{item3} $F$ has finite partial solutions in $K^N$ for all 
$N \gg 0$.
\item The ideal $[F] := \langle \{ \sigma^j(F) \mid j \in \mathbb{N}\} \rangle 
\subseteq K \{ x_1, \ldots, x_n \}$ does not contain $1$.
\item The ideal $[F]^* := \langle \{ \sigma^j(F) \mid j \in \mathbb{Z} \}\rangle 
\subseteq K \{ x_1, \ldots, x_n \}^*$ does not contain $1$.	
\item $F$ has a solution in some difference $K$-algebra.
\end{enumerate}
\end{thm}

In order to prove Theorem~\ref{thm:DNSS-small}, 
we will
extract two 
consequences of~\cite{Varshavsky}.
In Lemma~\ref{lem:existence_homomorphism} and~\ref{lem:weak_solution_modp},
 $\phi_s$  denotes the $s$-th power of the Frobenius automorphism. 

\begin{lem}\label{lem:existence_homomorphism}
 For every finitely generated difference subring $R$ of a difference field $K$,
   there exist a prime $p$, a positive integer $s$, and a difference homomorphism 
  $\psi \colon R \to \mathbb{F}$, where $\mathbb{F}$ is the algebraic closure of $\mathbb{F}_p$ considered as a difference ring
  with respect to $\phi_s$.
\end{lem}

\begin{proof}
  The proof will proceed in two steps.
  \begin{enumerate}[{Step} 1{:}\ ]
    \item We will show that there exists a prime $p$ and a difference field $L$ of characteristic $p$ such that there exists a homomorphism $R \to L$ of difference rings.
    If $\Char K >0$, then we can take $L$ to be $K$.
  Let now $\Char K=0$ and
    $R$ generated by $a_1, \ldots, a_\ell$. 
    Since $R$ is a difference subring of a difference field, the ideal 
    \begin{equation}\label{eq:defining_ideal_R}
      I := \{ f \in \mathbb{Z} \{ x_1, \ldots, x_\ell \} \:|\: f(a_1,\ldots,a_\ell) = 0 \}
    \end{equation}
    is a perfect difference ideal~\cite[p. 76, \S 12]{Cohn}.  
    As such, because every finitely generated difference ring is a Ritt difference ring~\cite[Chapter 3, Theorems~II, and~V]{Cohn}, 
    $I$ is finitely generated as a perfect difference ideal.
    Let $g_1, \ldots, g_s \in \mathbb{Z}\{x_1, \ldots, x_\ell\}$ be a finite set of such generators.
    Consider a model of $\operatorname{ACFA}_0$ containing $K$.
    Then the sentence
    \[
    \varphi:= \exists \mathbf{x} \; (g_1(\mathbf{x}) = \ldots = g_s(\mathbf{x}) = 0)
    \]
    is true in this model.
    \cite[(1.6), 2nd paragraph]{ChHr99} implies that there exists a finite disjunction, say $\psi$, of sentences specifying (up to an isomorphism) 
    a difference field structure
    on some Galois extensions of the prime subfield
    such that
    \begin{itemize}
        \item $\varphi$ and $\psi$ are equivalent in $\operatorname{ACFA}_0$. In particular, $\psi$ holds in some model of $\operatorname{ACFA}_0$.
        \item There exists a positive integer $N$ such that, for every prime $p > N$, $\varphi$ and $\psi$ are equivalent in $\operatorname{ACFA}_p$.
    \end{itemize}
  Applying the Chebotarev density theorem as in~\cite[(1.14)]{ChHr99}, one can show that, since $\psi$ is consistent with  $\operatorname{ACFA}_0$, there are inifinitely many primes $p$ such that $\psi$ holds in some model of  $\operatorname{ACFA}_p$.
    We fix such $p$ that is greater than $N$ and a model $L$ of $\operatorname{ACFA}_p$ in which $\psi$ and, consequently, $\varphi$ hold.
    Then there are $b_1, \ldots, b_\ell \in L$ such that $g_1(b_1, \ldots, b_\ell) = \ldots = g_s(b_1, \ldots, b_\ell) = 0$.
    Then the kernel of a difference homomorphism $\mathbb{Z}\{x_1, \ldots, x_\ell\} \to L$ defined by $x_i \mapsto b_i$ contains $I$, so it yields a difference homomorphism $R \to L$.
    \item If $\Char K = 0$, we replace $K$ with $L$ and $R$ with its image in $L$.
    Thus, in what follows, we assume that $\operatorname{char}{K} = p > 0$.
    Let $h$ be the maximum of the orders of $g_1, \ldots, g_s$.
    Let $b_1, \ldots, b_{N}$ be the elements of \[\{ \sigma^j (a_i) \:|\: 1\leqslant i \leqslant \ell, 0 \leqslant j < h \}\]
    written in some order, so $N = \ell h$.
    Then $R$ is also generated by $b_1, \ldots, b_N$ as a difference ring, and the corresponding vanishing ideal in the difference polynomial 
    ring $\mathbb{Z} \{ y_1, \ldots, y_N \}$ is generated as a perfect difference ideal by difference polynomials of order one.  
    Replacing $a_1, \ldots, a_\ell$ by $b_1, \ldots, b_N$, we may assume that $I$, defined in~\eqref{eq:defining_ideal_R}, is
    generated as a perfect difference ideal by order one difference polynomials.
 
    Let 
  $\mathfrak{q} \subseteq \mathbb{F}_p[x_1, \ldots, x_\ell, y_1, \ldots, y_\ell]$ be the ideal of all polynomials
  vanishing on $(a_1,\ldots,a_\ell, \sigma(a_1), \ldots, \sigma(a_\ell))$.
    Since $a_1, \ldots, a_\ell$ are elements of a difference field, the ideals
    \[
    \mathfrak{p}_1 := \mathfrak{q} \cap \mathbb{F}_p[\mathbf{x}] \quad\text{ and } \quad \mathfrak{p}_2 := \mathfrak{q} \cap \mathbb{F}_p[\mathbf{y}]
    \]
    are transformed one to the other under the substitution $\mathbf{x} \mapsto \mathbf{y}$. 
    Then 
  \[
    X := \operatorname{Spec}(\mathbb{F}_p[\mathbf{x}]/\mathfrak{p}_1) = \operatorname{Spec}(\mathbb{F}_p[\mathbf{y}]/\mathfrak{p}_2)\ \text{ and }\ C := \operatorname{Spec}(\mathbb{F}_q[\mathbf{x}, \mathbf{y}]/\mathfrak{q})
  \]
  are irreducible schemes of finite type over $\mathbb{F}_p$, and $C$ is a subset of $X \times X$.
  Then, by~\cite[Theorem 0.1]{Varshavsky}, there exists a positive integer $s$ such that the intersection of $C$ 
  with the graph of $\phi_s$ in $\mathbb{F}^{2\ell}$ is nonempty, where $\mathbb{F}$ is the algebraic closure of $\mathbb{F}_p$.  
  Let $(a_1^\ast, \ldots, a_\ell^\ast, \phi_s(a_1^\ast), \ldots, \phi_s(a_\ell^\ast))$ be a point in the intersection.
  Since the substitution $\sigma^j(x_i) = \phi_s^j(a^\ast_i)$ annihilates $\mathfrak{q}$, it also annihilates
  every polynomial in its perfect closure $I \pmod{p}$.
  Then the map $\psi\colon R \to \mathbb{F}$ defined by $\psi\left( \sigma^j(a_i) \right) = \phi_s^j(a_i^\ast)$
  is a desired homomorphism of difference rings $(R, \sigma)$ and $(\mathbb{F}, \phi_s)$.\qedhere
  \end{enumerate}
\end{proof}

\begin{lem}\label{lem:weak_solution_modp}
For every
\begin{itemize}
\item prime number $p$,
\item positive integer $s$,
\item scheme $X$ of finite type defined over $\mathbb{F}$, the algebraic closure of $\mathbb{F}_p$, 
\item irreducible subvariety $\Gamma \subset X \times \phi_s(X)$
   such that the projections to $X$ and $\phi_s(X)$ are dominant,
\end{itemize}
there exists
an infinite sequence $(a_i)_{i=-\infty}^\infty$ such that
$(a_i, a_{i + 1}) \in \phi_{si}(\Gamma)$ for every $i \in \mathbb{Z}$.
\end{lem}

\begin{proof}
Since $X$ and $\Gamma$ are defined over 
some finite subfield of 
$\mathbb{F}$, there is a 
positive integer $\ell$ with 
$\phi_{s\ell}(X) = X$ and $\phi_{s\ell}(\Gamma) = \Gamma$.  
Lemma~\ref{lem:component_fiber_product} implies
that there exists an irreducible 
component $\Xi$ of the fiber product
\[
\Gamma \times_{\phi_s(X)} 
\phi_s(\Gamma) 
\times_{\phi_{2s}(X)} 
\cdots \times_{\phi_{(\ell-1)s}(X)}
\phi_{(\ell-1)s}(\Gamma).
\]
such that the projections of $\Xi$
onto $\Gamma, \phi_s(\Gamma), \ldots, \phi_{(\ell-1)s}(\Gamma)$
are dominant.
We denote the projection $\Xi \to \phi_{si}(\Gamma)$ by $\rho_i$ for every $0 \leqslant i \leqslant \ell - 1$.

Let
$\tau_1\colon\Gamma \to X$ and 
$\tau_2\colon\Gamma \to \phi_s(X)$ denote the projections.  
We define projections $\pi_i\colon\Xi \to \phi_{s i}(X)$ for $0 \leqslant i \leqslant \ell$ as follows:
\[
\pi_i = \begin{cases}
  \tau_1 \circ \rho_0, \text{ for } i = 0,\\
  \tau_1 \circ \rho_i = \tau_2 \circ \rho_{i - 1}, \text{ for } 0 < i < \ell,\\
  \tau_2 \circ \rho_{\ell - 1}, \text{ for } i = \ell.
\end{cases}
\]

Note that the $\pi_i$'s are dominant.
Consider the fiber product of 
$\Xi \times_X \Xi$ where the 
first $\Xi \to X$ is
$\pi_\ell$ and the 
second map $\Xi \to X$ is 
$\pi_0$. Lemma~\ref{lem:component_fiber_product} implies
that there exists an irreducible component $\Upsilon$ of this product
such that the projections of $\Upsilon$ onto both $\Xi$'s are dominant.
Take $r$ so that $\Xi$ and $\Upsilon$
are both defined over $\mathbb{F}_{p^r}$ and \[s\ell \:|\: r.\] 
By~\cite[Theorem 0.1]{Varshavsky},
there is a power $\phi_t$ of $\phi_r$
and a point $a = (a_0, \ldots, a_{\ell - 1})\in \Xi(\mathbb{F})$ with $(a,\phi_{t}(a)) \in 
\Upsilon(\mathbb{F})$.
Note that $\phi_t$ leaves invariant $\Gamma$, $\Xi$, and $\Upsilon$.
Since coefficients of the $\pi_i$'s and $\rho_j$'s are invariant under $\phi_1$,
we will denote the conjugation of any of these maps by any power of $\phi_1$ by the same letter.
For $0 \leq i  < \ell$ and 
$j \in \mathbb{Z}$, define 
\[a_{i + j \ell} := \pi_i( \phi_{t j}(a)).\]

Let us show that the sequence $\{a_i\}_{i = -\infty}^\infty$ satisfies the requirement of the lemma.
Consider $j \in \mathbb{Z}$ and $0 \leqslant i < \ell$. Then $a_{i + j\ell} = \tau_1\left( \rho_i(\phi_{t j}(a)) \right)$.
We also have 
\begin{align*}
  &a_{i + j\ell + 1} = \tau_1\left( \rho_{i + 1}(\phi_{t j}(a)) \right) = \tau_2\left( \rho_i(\phi_{t j}(a)) \right) \text{ for } i < \ell - 1,\\
  &a_{i + j\ell + 1} = \tau_1\left( \rho_0(\phi_{t (j + 1)}(a))\right) = \tau_2\left( \rho_{\ell - 1}(\phi_{t j}(a)) \right) \text{ for } i = \ell  - 1
\end{align*} 
because $\Xi$ and $\Upsilon$ are components of the corresponding fiber products.
In both cases, 
\[
(a_{i + j \ell}, a_{i + j \ell + 1}) \in \rho_i\left( \phi_{t j}(\Xi) \right) = \rho_i(\Xi) \subset \phi_{si}(\Gamma) = \phi_{s(i + j \ell)}(\Gamma).\qedhere
\]
\end{proof}

In Lemmas~\ref{lem:multidimhensel} and~\ref{lem:smoothlift}, for a valued field $(K,v)$, we 
write \[\mathcal{O} = \{ x \in K ~:~ v(x) \geqslant 0 \}\] 
for the valuation ring, \[\mathfrak{m} = \{ x \in K ~:~
v(x) > 0 \}\] for the maximal ideal of $\mathcal{O}$, and
$k = \mathcal{O}/\mathfrak{m}$ for the residue field.
We denote the reduction map $r:\mathcal{O} \to k$ by $r$
and will abuse notation writing $r$ for the reduction 
map on associated objects.

\begin{lem}
\label{lem:multidimhensel}
Let $(K,v)$ be a Henselian field, $n \leqslant m$ positive integers, $f_1, \ldots, f_n \in \mathcal{O}[x_1, \ldots, x_m]$ 
and $a = (a_1, \ldots, a_m) \in k^m$.  
We assume that, for each $i$, we have $r(f_i)(a) = 0$ and
that the matrix $\left(\frac{\partial r(f_i)}{\partial x_j}(a)\right)_{1 \leqslant i
\leqslant n,\, 1 \leqslant j \leqslant m}$ has rank $n$.  
Then there is $c = (c_1, \ldots, c_m) \in \mathcal{O}^m$ such that
$f_1(c) = \ldots = f_n(c) = 0$ and $r(c) = a$.
\end{lem}

\begin{proof}
By hypothesis, there is some $J \subseteq \{ x_1, \ldots, x_m \}$ 
with $|J| = n$ and invertible matrix $\left(\frac{\partial r(f_i)}{\partial x_j}(a) \right)_{1 \leqslant i
\leqslant n,\, j \in J}$.  Relabeling the variables, we may 
assume that $J = \{ 1, \ldots, n\}$.   Define $f_i := x_i$ 
for $n < i \leq m$.   Then the square matrix 
$\left(\frac{\partial r(f_i)}{\partial x_j}(a)  \right)_{1 \leqslant i
\leqslant n,\, 1 \leqslant j \leqslant n}$ is invertible.
There exists some $b \in \mathcal{O}^m$ such that $r(b) = a$.
Then, by to~\cite[Section~4, Multidimensional Hensel's Lemma]{Ku2}, there is some $c \in \mathcal{O}^n$ with 
$f_1(c) = \cdots = f_n(c) = 0$ and $r(c) = r(a)$.
\end{proof}

\begin{lem}
\label{lem:smoothlift}
Let $(K,v)$ be a Henselian field and 
$f:X \to Y$  a smooth map of schemes of
finite type over $\mathcal{O}$.   Suppose that 
$a \in X(k)$ and $b \in Y(\mathcal{O})$ 
satisfy $f(a) = r(b)$.  Then there is a 
point $c \in X(\mathcal{O})$ with $f(c) = b$ and
$r(c) = a$.
\end{lem}

\begin{proof}
\cite[\href{https://stacks.math.columbia.edu/tag/01V7}{Tag 01V7}]{stacks-project} implies that
there are affine open neighborhoods 
$U \subseteq X$ and $V \subseteq Y$ of
$a$ and $r(b)$, respectively, for which 
$f_U\colon U \to V$ is 
standard smooth.  That is, there exist:
\begin{itemize}
\item positive integers $m$ and $n$,
\item a finitely generated $\mathcal{O}$-algebra $S$ such that $V = \operatorname{Spec}(S)$, 
\item polynomials $g_1,\ldots,g_n \in S[x_1,\ldots,x_m]$ such that $U = \operatorname{Spec}(T)$, where $T = S[x_1,\ldots,x_m]/(g_1,\ldots,g_n)$,
\end{itemize}
such that
\begin{itemize}
\item
some $n\times n$ minor of
the Jacobian $(\frac{\partial g_i}{\partial x_j})$ is an invertible element of $T$ and
\item $f_U$ is the dual morphism of schemes to the natural homomorphism
$S \to T$.
\end{itemize}
Since $\mathcal{O}$ is a local ring and $r(b)$ is a reduction of $b$ modulo $\mathfrak{m} \subset \mathcal{O}$,
the point $b$ belongs to any open neighborhood of $r(b)$, in particular, $b \in V(\mathcal{O})$. 
This corresponds
to an $\mathcal{O}$-algebra homomorphism $b^\sharp\colon S \to \mathcal{O}$. 
Let $a^\sharp$ denote the $\mathcal{O}$-algebra homomorphism $\mathcal{O}[U]\to k$ corresponding to $a \in U(k)$.

For each $i$, $1\leqslant i\leqslant n$,
consider the polynomials $b^\sharp(g_i) \in \mathcal{O}[x_1, \ldots, x_m]$ and 
$a^\sharp(g_i) \in k[x_1, \ldots, x_m]$ that are obtained from $g_i$ by applying $b^\sharp$ and $a^\sharp$, respectively, to the coefficients. 
The fact $r(b) = f(a)$ implies that
$r\left( b^\sharp(g_i) \right) = a^\sharp(g_i)$.
Let $a_j$ be the result of applying $a^\sharp$ to the image of $x_j$ in $T$ for $1 \leqslant j \leqslant m$.
Since $a^\sharp$ is a homomorphism, $a^\sharp(g_i)(a_1, \ldots, a_m) = 0$ for every $1 \leqslant i \leqslant n$
and also the Jacobian matrix $\left(\frac{\partial a^\sharp(g_i)}{\partial x_j}\right)$ has full rank at $(a_1, \ldots, a_m)$.

Then, by Lemma~\ref{lem:multidimhensel}, we may find $(c_1, \ldots, c_m)
\in \mathcal{O}^m$ for which $b^\sharp(g_i)(c_1, \ldots, c_m) = 0$ for $1 \leqslant i \leqslant n$ 
and $r(c_j) = a_j$ for $1 \leqslant j \leqslant m$.
Since $b^\sharp(g_i)(c_1, \ldots, c_m) = 0$ for $1 \leqslant i \leqslant n$,
the map $c^\sharp\colon T \to \mathcal{O}$ defined by $c^\sharp|_S = b^\sharp$ and 
by $c^\sharp(x_i) = c_i$ for $1 \leqslant i \leqslant m$ is a well-defined $\mathcal{O}$-algebra  homomorphism.
This gives us a point $c \in U(\mathcal{O})$ such that $f(c) = b$.
Moreover, $r(c) = a$ since $r(c_i) = a_i$ for every $1 \leqslant i \leqslant m$.
\end{proof}

\begin{cor}\label{cor:lifting}
 Let $(K, v)$ be a Henselian field and $X$ a scheme of finite type over $\mathcal{O}$
 such that the canonical morphism $X \to \operatorname{Spec}(\mathcal{O})$ is smooth.
 Then, for every $a \in X(k)$, there exists $c \in X(\mathcal{O})$ such that $r(c) = a$.
\end{cor}

\begin{proof}
  The corollary follows from Lemma~\ref{lem:smoothlift} applied to $Y = \operatorname{Spec}(\mathcal{O})$ and $b$ being
  the identity map $\operatorname{Spec}(\mathcal{O}) \to \operatorname{Spec}(\mathcal{O})$.
\end{proof}

With these statements in place, 
we finish the proof of
Theorem~\ref{thm:DNSS-small}.

In what follows, for a positive integer $m$ and a commutative ring $R$, $R^m$ denotes the commutative ring generated by the set $\{r^m\:|\: r \in R\}$.
For an affine scheme $X$ over a perfect ring $R$ of characteristic $p$ and $q = p^n$,
we define a scheme $X^{(q)}$ by $X^{(q)} := \operatorname{Spec}(\mathcal{O}_X^q)$.
There is a map $F_n\colon X \to X^{(q)}$ that is dual to the inclusion $\mathcal{O}_X^q \hookrightarrow \mathcal{O}_X$.
This map is a special case of what is called the relative Frobenius morphism.  See~\cite[\href{https://stacks.math.columbia.edu/tag/0CC6}{Tag 0CC6}]{stacks-project} for more
details.
If $R$ is perfect, $F_n$ defines a bijection between $X(R)$ and $X^{(q)}(R)$.
If $p=0$, we will assume that $q = 1$ and $F_n$ is the identity map.

\begin{lem}
  \label{lem:insepfactor}
  If $\mu:\Gamma \to Z$ is morphism of irreducible affine varieties
  over an algebraically closed field $K$, then there exist
  \begin{itemize}
    \item an affine variety $\Upsilon$,
    \item morphisms  $\nu\colon \Gamma \to \Upsilon$ and $\tau\colon \Upsilon \to Z$,
    \item  a positive integer $n$ and   a morphism $\gamma:\Upsilon \to \Gamma^{(q)}$, where $q = p^n$,
  \end{itemize}
  such that $\mu = \tau \circ \nu$, $\gamma \circ \nu = F_n$, $\nu$ is finite, and $\tau$ is generically smooth.
\end{lem}
\begin{proof}
  If $\operatorname{char} K=0$,  take $\Upsilon = \Gamma$, $\nu = \operatorname{id}_\Gamma$ and $\tau = \mu$ by \cite[Theorem~2.27]{Shafarevich}.
  
 Let $\operatorname{char}K = p > 0$,  
  $t_1, \ldots, t_\ell$ be a transcendence basis of $K(\Gamma)$ over $E := \operatorname{Quot}(\mu^\ast(\mathcal{O}_Z))$, and 
   $L$ be the relative separable closure of $E(t_1, \ldots, t_\ell)$
  in $K(\Gamma)$.  Then, as $K(\Gamma)$ is a finite purely inseparable extension of
  $L$, for $n \gg 0$ we have $K(\Gamma)^{p^n} \subseteq L$.  Let $q := p^n$ and 
$A = \mu^\ast(\mathcal{O}_Z) [\mathcal{O}_\Gamma^q]$, the 
  ring generated by $\mu^\ast(\mathcal{O}_Z)$ and $\mathcal{O}_\Gamma^{q}$.   
  Set $\Upsilon := \operatorname{Spec}(A)$ over $K$.
  
  Dual to the homomorphisms of rings $\mathcal{O}_Z \to A$ and $A \to \mathcal{O}_\Gamma$,
  we have morphisms $\tau\colon \Upsilon \to Z$ and $\nu\colon \Gamma \to \Upsilon$ with $\mu = \tau \circ \nu$.  
  Since the field extensions $E \hookrightarrow K(\Upsilon)$ is a subextension of the 
  the separable extension $E \hookrightarrow L$, the morphism $\tau\colon \Upsilon \to Z$ is
  smooth at the generic point of $\Upsilon$ due to~\cite[\href{https://stacks.math.columbia.edu/tag/07ND}{Tag 07ND}]{stacks-project}.  
  Form the inclusion $\mathcal{O}_\Gamma^{q}  \hookrightarrow A = \mathcal{O}_\Upsilon$, 
  we obtain the morphism $\gamma\colon  \Upsilon\to \Gamma^{(q)}$ with $F_n = \gamma \circ \nu$. 
  
  Since $\mathcal{O}_\Gamma^q \subset \mathcal{O}_\Gamma$ is a finite integral extension and $\mathcal{O}_\Gamma^q \subset A \subset \mathcal{O}_\Gamma$, the extension $A \subset \mathcal{O}_\Gamma$ is also a finite integral extension.
  Hence, the dual map $\nu\colon \Gamma \to \Upsilon$ is a finite morphism.
\end{proof}

\begin{proof}[Proof of Theorem~\ref{thm:DNSS-small}]

The only implication whose
proof in the original argument 
for Proposition~\ref{thm:DNSS}
used uncountability is from 5. to 1.
We observe that 5. implies 3., because $1$ is not contained in any ideal generated by finitely many transforms of the system, so Hilbert's Nullstellensatz implies that there exist arbitrarily long partial solutions of the system over~$K$. This is exactly~3. 

Consider the triple $(X, \pi_1, \pi_2)$ constructed in Section~\ref{subsec:var_proj}.
Due to Lemma~\ref{lem:sol_geomertic}, item~\ref{item3} implies that $(X, \pi_1, \pi_2)$ has  arbitrarily long partial solutions.
On the other hand, the existence of a solution to $F = 0$ in $K^\mathbb{Z}$ is 
equivalent to the existence of a two-sided infinite solution to 
$(X,\pi_1,\pi_2)$ over $K$ (see Lemma~\ref{lem:sol_geomertic}).
We thus reduce to finding a 
solution to $(X,\pi_1,\pi_2)$ over $K$. 
Then Proposition~\ref{prop:bound_general} implies that there exists an infinite skew-cyclic train 
\begin{equation}
  \label{eq:skewcyc}\left( \ldots, Y_1, Y_2, \ldots, Y_\ell, \sigma^\ell(Y_1), \ldots \right)
\end{equation} in $X$.
Let $\mathbf{c}$ be a signature of the train $(Y_1, \ldots, Y_\ell)$ and let $Y \subset W_{\mathbf{c}} \subset X^{\mathbf{c}}$ be the associated irreducible variety given by Lemma~\ref{lem:train_to_variety}.
For $1 \leq i \leq \ell$, 
let  $\rho_i\colon Y \to Y_i$ be
the dominant projection to $Y_i$ (which is ${\psi_{\ell,i}|}_Y$ in the notation of Lemma~\ref{lem:train_to_variety}).
The projection $\sigma^j(Y) \to \sigma^j(Y_i)$ obtained by conjugation by $\sigma^j$ of $\rho_i$ will be denoted by $\sigma^j(\rho_i)$ for every $j \in \mathbb{Z}$.
Recall that, since~\eqref{eq:skewcyc}
is a train,
$\pi_2 \circ 
\rho_\ell$ and $\pi_1 \circ \sigma^\ell(\rho_1)$
are dominant onto the same variety.
Due to Lemma~\ref{lem:component_fiber_product}, there exists an irreducible
component $\Gamma$, which we fix, of the fiber 
product of $Y$ with 
$\sigma^\ell(Y)$ over $\pi_2 \circ \rho_\ell$ and $\pi_1 \circ \sigma^\ell(\rho_1)$ 
such that $\mu_1:\Gamma \to Y$ and $\mu_2:\Gamma \to \sigma^\ell(Y)$ are dominant.  

Let us call a sequence
$(a_i)_{i = -\infty}^\infty$ with 
$a_i \in \sigma^{i\ell}(Y)$ and $(a_i, a_{i + 1}) \in \sigma^{i \ell}(\Gamma)$
for all $i$
a \emph{weak solution} to $(Y,\Gamma)$. 
Such a weak solution gives rise 
to the solution 
\[(\ldots, \sigma^{-\ell}(\rho_1)(a_{-1}), \ldots, \sigma^{-\ell}(\rho_\ell)(a_{-1}), \rho_1(a_0), \rho_2(a_0), \ldots, \rho_{\ell}(a_0), 
\sigma^\ell(\rho_1)(a_1), \ldots, \sigma^\ell(\rho_\ell)(a_1), \ldots)\] 
of $(X,\pi_1,\pi_2)$.  Thus, 
it suffices for us to find 
a weak solution.
Lemma~\ref{lem:insepfactor} implies that there exist
\begin{itemize}
\item $\Upsilon_1$ and $\Upsilon_2$ be affine varieties
\item  $n$ a positive integer,
\item  $\nu_i\colon\Gamma \to \Upsilon_i$, 
$\tau_i\colon \Upsilon_i \to Y$,  
$\gamma_i\colon\Upsilon_i \to \Gamma^{(q)}$, morphisms, where $q = p^n$ and $i=1,2$,
\end{itemize}
so 
that, for $i=1,2$,
\begin{itemize}
\item $\gamma_i \circ \nu_i = F_n$,
\item $\tau_i$ is 
generically smooth, 
\item $\mu_i = \tau_i \circ \nu_i$.
\end{itemize}

We fix some equations defining $\Gamma$, $Y$,  $\Upsilon_i$, $\gamma_i$, $\tau_i$, $\nu_i$, and $\mu_i$ for $i = 1, 2$.
Denote the difference ring generated by the coefficients of these equations by $R$.
Let $\pi\colon \Gamma^{(q)} \to \operatorname{Spec}(R)$ be the dual to the natural embedding $R \to \mathcal{O}_{\Gamma}^q$.
\cite[\href{https://stacks.math.columbia.edu/tag/07ND}{Tag 07ND}]{stacks-project} implies that $\pi$ is generically smooth.
Let $\mathsf{\Gamma}$, $\mathsf{Y}$, $\mathsf{\Upsilon}_1$, and $\mathsf{\Upsilon}_2$ be the models of
$\Gamma$, $Y$, $\Upsilon_1$, and $\Upsilon_2$ defined by these fixed equations over $R$.
Thus, we have the following diagram:
$$
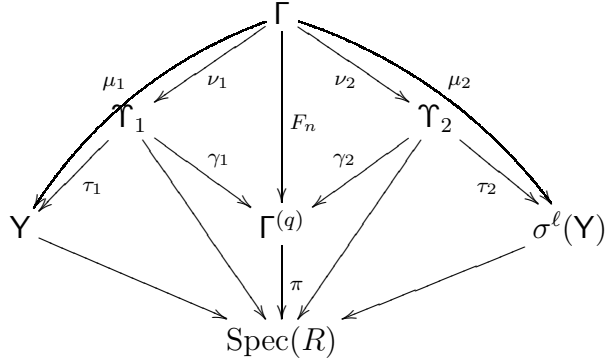
\xymatrix{  & & \mathsf{\Gamma} \ar[dd]^{F_n}   \ar[ld]^{\mathsf{\nu}_1} \ar[rd]_{\mathsf{\nu}_2} 
 \ar@/_1pc/[lldd]_{\mathsf{\mu}_1}  \ar@/^1pc/[rrdd]^{\mathsf{\mu}_2} & & \\ 
 & \mathsf{\Upsilon}_1 \ar[rdd] \ar[ld]^{\mathsf{\tau}_1} \ar[rd]^{\mathsf{\gamma}_1} & & \mathsf{\Upsilon}_2 \ar[ldd] \ar[rd]_{\mathsf{\tau}_2}  \ar[ld]_{\mathsf{\gamma}_2} & \\
\mathsf{Y} \ar[rrd] & & \mathsf{\Gamma}^{(q)} \ar[d]^{\pi} & & \mathsf{\sigma^\ell(Y)} \ar[lld] \\
 & & \operatorname{Spec}(R) & &  
}
$$
Let $\hat{\mathsf{\Upsilon}}_1$, $\hat{\mathsf{\Upsilon}}_2$, and $\hat{\mathsf{\Gamma}}^{(q)}$ be
dense open subsets in $\mathsf{\Upsilon}_1$, $\mathsf{\Upsilon}_2$, and $\mathsf{\Gamma}^{(q)}$, respectively,
such that $\tau_1$, $\tau_2$, and $\pi$, respectively, are smooth on these subsets, which exist since smoothness of a morphism is an open condition 
(see the discussion just after~\cite[\href{https://stacks.math.columbia.edu/tag/01V5}{Tag 01V5}]{stacks-project}).
Let
\[\widetilde{\mathsf{\Gamma}}=\nu_1^{-1}(\hat{\mathsf{\Upsilon}}_1)\cap\nu_2^{-1}(\hat{\mathsf{\Upsilon}}_1)\cap F_n^{-1}(\hat{\mathsf{\Gamma}}^{(q)}),
\]
which is dense open in $\mathsf{\Gamma}$.
Let $\mathsf{\Gamma}'$ be a non-empty open subset of $\widetilde{\mathsf{\Gamma}}$ 
defined by a single inequality $f \neq 0$, 
where $f \in \mathcal{O}_{\mathsf{\Gamma}}$.
The image of $\mathsf{\Gamma}'$ under $F_n$ is open dense in 
$(\mathsf{\Gamma}')^{(q)} \subset \mathsf{\Gamma}^{(q)}$, defined by $f^q \neq 0$.
Let \[\mathsf{\Upsilon}_i'=\gamma_i^{-1}\left((\mathsf{\Gamma}')^{(q)}\right) \cap \hat{\Upsilon}_i,\ \ i=1,2.\]
Then $\nu_i\left(\mathsf{\Gamma}'\right)\subset \mathsf{\Upsilon}_i'$,   and 
$(\mathsf{\Gamma}')^{(q)} \subset \hat{\mathsf{\Gamma}}^{(q)}$.

We apply Lemma~\ref{lem:existence_homomorphism} to $(R, \sigma^\ell)$ and obtain $\psi\colon (R,\sigma^\ell) \to (\mathbb{F},\phi_s)$,
where $\mathbb{F}$ is the algebraic closure of $\mathbb{F}_p$ and, in the case
$\operatorname{char} K = 0$, $p$ is some prime number provided by Lemma~\ref{lem:existence_homomorphism}.
Let $X_{\mathbb{F}}$  denote the base change of a scheme $X$ over $R$ to $\mathbb{F}$ via $\psi$.
Let
$(a_i)_{i=-\infty}^\infty$ be a sequence such that, for each $i\in\mathbb{Z}$, 
\[
  (a_i,a_{i+1}) \in \phi_{si}(\mathsf{\Gamma}'_{\mathbb{F}})(\mathbb{F}).
\] 
Such a sequence exists by Lemma~\ref{lem:weak_solution_modp}.  
Fix an extension of $\psi$ to a place $\vartheta$ on $K$ (see~\cite[Theorem~3.1.1]{ValuedFields}).
Let $\mathcal{O}$ be the  valuation ring of $\vartheta$ and $v$ be a valuation on $K$. Note that $R \subset \mathcal{O}$.  
Also note that we do not assert that $\vartheta$ respects $\sigma$ on all of $\mathcal{O}$ nor even that $\mathcal{O}$ is preserved by $\sigma$.
Let $\mathbb{E}$ be the residue field of $\mathcal{O}$.
Since $K$ is algebraically closed, $\mathbb{E}$ is also algebraically closed~\cite[Theorem~3.2.11]{ValuedFields}.
Since $\mathbb{F}_p \subset \mathbb{E}$,  $\mathbb{F}$ is embedded into $\mathbb{E}$.

\cite[\href{https://stacks.math.columbia.edu/tag/01VB}{Tag 01VB}]{stacks-project}~implies that the morphisms of schemes 
$(\mathsf{\Upsilon}'_1)_\mathcal{O} \to {\mathsf{Y}}_{\mathcal{O}}$, 
$(\mathsf{\Upsilon}'_2)_{\mathcal{O}} \to \sigma^\ell(\mathsf{Y})_{\mathcal{O}}$, 
and $(\mathsf{\Gamma}')^{(q)}_{\mathcal{O}} \to \operatorname{Spec}(\mathcal{O})$
are smooth as well as all their shifts/conjugations by $\sigma^{\ell}$.
We shall now build a weak solution $(b_i)_{i = -\infty}^\infty$ to $\left(\mathsf{Y}_{\mathcal{O}}({\mathcal{O}}),\mathsf{\Gamma}'_{\mathcal{O}}({\mathcal{O}})\right)$ 
so that \[\forall\, i \in\mathbb{Z} \ \ \vartheta(b_i) = a_i.
\]
Since $K$ is algebraically closed, $(K, v)$ is Henselian \cite[Lemma~4.1]{Ku2}.
For $i = 0$ and $i = 1$, since 
$\pi\colon(\mathsf{\Gamma}')^{(q)}_{\mathcal{O}} \to \operatorname{Spec}(\mathcal{O})$ is smooth, 
every point in $(\mathsf{\Gamma}')^{(q)}_{\mathbb{F}}(\mathbb{F})$ 
lifts to a point in $(\mathsf{\Gamma}')^{(q)}_{\mathcal{O}} (\mathcal{O})$ 
due to Corollary~\ref{cor:lifting}.
Thus, we may choose some $(\hat{b}_0,\hat{b}_1) \in (\mathsf{\Gamma}')^{(q)}_{\mathcal{O}} (\mathcal{O})$ 
specializing to $(F_n(a_0), F_n(a_1))$ and set $b_0 = F_n^{-1}(\hat{b}_0)$ and $b_1 = F_n^{-1}(\hat{b}_1)$.   
 
Assume that we have already constructed $b_i$ for some $i > 0$ so that $\vartheta(b_i) = a_i$.
Due to Lemma~\ref{lem:smoothlift} applied to the morphism of schemes 
\[\sigma^{i\ell}\circ\tau_1\circ\sigma^{-i\ell}: \sigma^{i\ell}\left( (\mathsf{\Upsilon}'_1)_{\mathcal{O}}\right) \to \sigma^{i\ell}\left(\mathsf{Y}_{\mathcal{O}} \right)\] 
and points $(\sigma^{i\ell}\circ\nu_1\circ\sigma^{-i\ell})\left((a_i, a_{i + 1})\right)$ and $b_i$,
there exists $P \in \sigma^{i\ell} \left( (\mathsf{\Upsilon}'_1)_{\mathcal{O}} \right)$ such that \[(\sigma^{i\ell}\circ\tau_1\circ\sigma^{-i\ell})(P) = b_i\quad 
\text{and}\quad \vartheta(P)=(\phi_s^{i}\circ\nu_1\circ\phi_s^{-i})\left((a_i, a_{i + 1})\right).\]
Consider \[Q = F_n^{-1}\left(\sigma^{i\ell}\circ \gamma_1\circ\sigma^{-i\ell}(P) \right) \in \sigma^{i\ell} \left( \mathsf{\Gamma}'_{\mathcal{O}} (\mathcal{O}) \right).\]
Since $\nu_1$ is a finite 
morphism, it is surjective on $\mathcal{O}$-points
due to~\citep[Theorem~1.12]{Shafarevich} together with~\citep[Theorem~3.1.3]{ValuedFields}.  
Using this and the fact that $F_n$ is bijective
on $\mathcal{O}$-points, $\sigma^{i\ell} \circ \nu_1 \circ \sigma^{-i\ell} (Q) = P$.
Hence, \[(\sigma^{i\ell}\circ\mu_1\circ\sigma^{-i\ell})(Q) = (\sigma^{i\ell}\circ\tau_1\circ\sigma^{-i\ell})(P) = b_i,\] so $Q$ can be written as $(b_i, c)$. Since \[F_n^{-1} \circ \phi_s^{i}\circ\gamma_1 \circ \nu_1\circ\phi_s^{-i}=\phi_s^{i}\circ F_n^{-1} \circ\gamma_1 \circ \nu_1\circ\phi_s^{-i} = \phi_s^i \circ \operatorname{id} \circ \phi_s^{-i} = \operatorname{id},\] we have
\[\vartheta(Q)=F_n^{-1} \circ \phi_s^{i}\circ\gamma_1 \circ \nu_1\circ\phi_s^{-i} \left( (a_i, a_{i + 1}) \right) = (a_i, a_{i + 1}).\]
Thus, we can set $b_{i + 1} = c$.
In the same way, we produce the $b_i$ with $i < 0$  using the fact that $(\mathsf{\Upsilon}'_2)_{\mathcal{O}} \to \sigma^\ell(\mathsf{Y})_{\mathcal{O}}$ is smooth.
\end{proof}

\bigskip
\footnotesize
\noindent\textit{Acknowledgments.}
This work has been partially supported by the NSF grants CCF-0952591, CCF-1563942,  DMS-1413859, DMS-1363372, DMS-1760413, DMS-1760448, by the NSA grant \#H98230-15-1-0245, by PSC-CUNY grant \#60098-00 48, by Queens College Research Enhancement, and by the Austrian Science Fund FWF grant Y464-N18. The authors are grateful to the CCiS at CUNY Queens College for the computational resources and to the referees for their helpful comments.

\end{document}